\documentclass[12pt,leqno]{amsart}

\usepackage{amssymb}
\usepackage{amsthm}
\usepackage{amsmath}
\usepackage{amscd}
\usepackage[mathscr]{eucal}
\usepackage{verbatim}

\numberwithin{equation}{section}

\theoremstyle{plain}
\newtheorem{theorem}{Theorem}[section] 

\newtheorem{lemma}[theorem]{Lemma}
\newtheorem{proposition}[theorem]{Proposition}

\theoremstyle{definition}
\newtheorem{definition}[theorem]{Definition}
\newtheorem{remark}[theorem]{Remark}
\newtheorem{example}[theorem]{Example}

\theoremstyle{remark}

\newtheorem{problem}[theorem]{Problem}

\newcommand{\R}{\mathbb{R}}
\newcommand{\Q}{\mathbb{Q}}
\newcommand{\Z}{\mathbb{Z}}
\newcommand{\N}{\mathbb{N}}
\newcommand{\C}{\mathbb{C}}
\newcommand{\h}{\mathbb{H}}
\renewcommand{\H}{\mathbb{H}}

\newcommand{\G}{\Gamma}
\newcommand{\g}{\gamma}

\newcommand{\la}{\lambda}
\newcommand{\La}{\Lambda}

\newcommand{\back}{\backslash}

\newcommand{\wwedge}[1]{\sideset{}{^{#1}}\bigwedge}

\newcommand{\kzxz}[4]{\left(\begin{smallmatrix} #1 & #2 \\ #3 & #4\end{smallmatrix}\right) }

\newcommand{\calA}{\mathcal{A}}

\newcommand{\calF}{\mathcal{F}}

\newcommand{\calL}{\mathcal{L}}

\newcommand{\calS}{\mathcal{S}}

\newcommand{\eps}{\varepsilon}

\newcommand{\vol}{\operatorname{vol}}
\newcommand{\tr}{\operatorname{tr}}

\newcommand{\sgn}{\operatorname{sgn}}

\newcommand{\Span}{\operatorname{span}}
\newcommand{\Stab}{\operatorname{Stab}}

\newcommand{\Sl}{\operatorname{SL}}

\newcommand{\SL}{\operatorname{SL}}

\newcommand{\Spin}{\operatorname{Spin}}

\newcommand{\Orth}{\operatorname{O}}

\newcommand{\SO}{\operatorname{SO}}

\newcommand{\supp}{\operatorname{supp}}

\newcommand{\Lk}{\operatorname{Lk}}
\newcommand{\PD}{\operatorname{PD}}

\begin{document}

\title[The Geometric Theta Correspondence for Hilbert Modular Surfaces]
{The Geometric Theta Correspondence for Hilbert Modular Surfaces }

\author[Jens Funke and John Millson]{Jens Funke* and John Millson**}
\thanks{* Partially supported by NSF grant DMS-0710228}
\thanks{** Partially supported by NSF grant DMS-0907446, NSF FRG grant DMS-0554254, and the Simons Foundation}
\address{Department of Mathematical Sciences, University of Durham, Science Laboratories,
South Rd, Durham DH1 3LE, United Kingdom}
\email{jens.funke@durham.ac.uk}
\address{Department of Mathematics, University of Maryland, College Park, MD
20742, USA} \email{jjm@math.umd.edu}

\date{\today}

\maketitle

\section{Introduction}

In a series of papers \cite{FM1,FMcoeff,FMres,FMspec} we have been studying the geometric theta correspondence (see below) for non-compact arithmetic quotients of symmetric spaces associated to orthogonal groups. It is our overall goal to develop a general theory of geometric theta liftings in the context of the real differential geometry/topology of non-compact locally symmetric spaces of orthogonal and unitary groups which generalizes the theory of Kudla-Millson in the compact case, see \cite{KM90}. 

In this paper we study in detail the geometric theta lift for Hilbert modular surfaces. In particular, we will give a new proof and an extension (to all finite index subgroups of the Hilbert modular group) of the celebrated theorem of Hirzebruch and Zagier \cite{HZ} that the generating function for the intersection numbers of the Hirzebruch-Zagier cycles is a classical modular form of weight $2$.\footnote{Eichler, \cite{HZ} p.104, proposed a proof using ``Siegel's work on indefinite theta functions''. This is what our proof is, though with perhaps more differential geometry than Eichler had in mind.} In our approach we replace Hirzebuch's smooth complex analytic compactification $\tilde{X}$ of the Hilbert modular surface $X$ with the (real) Borel-Serre compactification $\overline{X}$. The various algebro-geometric quantities that occur in \cite{HZ} are then replaced by topological quantities associated to $4$-manifolds with boundary. In particular, the ``boundary contribution'' in \cite{HZ} is replaced by sums of linking numbers of circles (the boundaries of the cycles) in the $3$-manifolds of type Sol (torus bundle over a circle) which comprise the Borel-Serre boundary.

\subsubsection*{The geometric theta correspondence}

We first explain the term ``geometric theta correspondence''. The Weil (or oscillator) representation gives us a method to construct closed differential forms on locally
symmetric spaces associated to groups which belong to dual pairs. Let $V$ be a rational quadratic space of signature $(p,q)$ with for simplicity even dimension. Then the Weil representation induces an action of $\SL_2(\R) \times \Orth(V_\R)$ on $\mathcal{S}(V_\R)$, the Schwartz functions on $V_\R$. Let $G = \SO_0(V_\R)$ and let $K$ be a maximal compact subgroup. We let $\mathfrak{g}$ and $\mathfrak{k}$ be their respective Lie algebras and let $\mathfrak{g} = \mathfrak{p} \oplus \mathfrak{k}$ be the associated Cartan decomposition. Suppose 
\[
\varphi\in (\mathcal{S}(V_\R) \otimes \wedge^r \mathfrak{p}^{\ast})^{K}
\]
%\pagebreak
is a cocycle in the relative
Lie algebra complex for $G$ with values in $\mathcal{S}(V)$. Then $\varphi$ corresponds to a closed differential $r$-form $\tilde{\varphi}$ on the symmetric space $D= G/K$ of dimension $pq$ with values in $\mathcal{S}(V)$.
For a coset of a lattice $\mathcal{L}$ in $V$, we define the theta distribution $\Theta=\Theta_{\mathcal{L}}$ by $\Theta = \sum_{\ell \in \mathcal{L}} \delta_{\ell}$, where $ \delta_{\ell}$ is the delta measure concentrated at $\ell$. It is obvious that $\Theta$ is invariant under $\G = \Stab(\mathcal{L}) \subset G$. There is also a congruence subgroup $\Gamma'$ of $\SL(2,\Z)$) such that $\Theta$ is also invariant under $\Gamma'$. Hence we can apply the theta distribution to $\tilde{\varphi}$ to obtain a closed $r$-form $\theta_{\varphi}$ on $X = \Gamma \backslash D$ given by 
\[
\theta_{\varphi}(\mathcal{L})= \langle \Theta_{\mathcal{L}}, \tilde{\varphi} \rangle.
\]
Assume now in addition that $\varphi$ has weight $k$ under the maximal compact subgroup $\SO(2) \subset \SL_2(\R)$. Then $\theta_{\varphi}$ also gives rise to a (in general) non-holomorphic function on the upper half place $\h$ which is modular of weight $k$ for $\G'$. We may then use $\theta_{\varphi}$ as the kernel of a pairing of modular forms $f$ with (closed) differential $(pq-r)$-forms $\eta$ or $r$-chains (cycles) $C$ in $X$. The resulting pairing in $f$, $\eta$ (or $C$), and $\varphi$ as these objects vary, we call the {\bf geometric theta correspondence}.

\subsubsection*{The cocycle of Kudla-Millson}

The key point of the work of Kudla and Millson \cite{KM1,KM2} is that they found (in greater generality) a family of cocycles $\varphi^V_{q}$ in $(\mathcal{S}(V) \otimes \wedge^q \mathfrak{p}^{\ast})^K$ with weight $(p+q)/2$ for $\SL_2$. Moreover, these cocycles give rise to Poincar\'e dual forms for certain totally geodesic, ``special'' cycles in $X$. Recently, it has now been shown, first \cite{HoffmanHe} for $\SO(3,2)$, and then \cite{BMM} for all $\SO(p,q)$ and $p+q>6$ (with $p \geq q$) in the cocompact (standard arithmetic) case that the geometric theta correspondence specialized to $\varphi_q^V$ induces on the adelic level an {\it isomorphism} from the appropriate space of classical modular forms to $H^q(X)$. In particular, for any congruence quotient, the dual homology groups are spanned by special cycles. This gives further justification to the term geometric theta correspondence and highlights the significance of these cocycles. In \cite{FMcoeff} we generalize $\varphi^V_{q}$ to allow suitable non-trivial coefficient systems (and one has an analogous isomorphism in \cite{BMM}).

\subsubsection*{The main results}

In the present paper, we consider the case when $V$ has signature $(2,2)$ with $\Q$-rank $1$. Then $D \simeq \h \times \h$, and $X$ is a Hilbert modular surface. We let $\overline{X}$ be the Borel-Serre compactification of $X$ which is obtained by replacing each isolated cusp associated to a rational parabolic $P$ with a boundary face $e'(P)$ which turns out to be a torus bundle over a circle, a $3$-manifold of type Sol. This makes $\overline{X}$ a $4$-manifold with boundary.  For simplicity, we assume that $X$ has only one cusp so that $\partial \overline{X} = e'(P)$, and we write $k: \partial \overline{X}  \hookrightarrow \overline{X}$ for the inclusion. The special cycles $C_n$\footnote{We distinguish the relative cycles $C_n$ in $X$ from the Hirzebruch-Zagier cycles $T_n$ in $\tilde{X}$, see below.} in question are now embedded modular and Shimura curves, and are parameterized by $n \in \N$. They define relative homology classes in $H_2(X, \partial X,\Q)$. 

The geometric theta correspondence of Kudla-Millson \cite{KM90} for the cocycle $\varphi^V_{2}$ in this situation takes the following shape. For a {\it compact} cycle $C$ in $X$, we have that 
\begin{equation}\label{KM-id}
\langle \theta_{\varphi^V_{2}}, C \rangle =  \int_C  \theta_{\varphi^V_{2}}= \sum_{n \geq 0} (C_n \cdot C) q^n
\end{equation}
is a holomorphic modular form of weight $2$ and is equal to the generating series of the  intersection numbers with $C_n$. Here $q = e^{2\pi i \tau}$ with $\tau \in \h$. 
(There is a similar statement for the pairing of $\theta_{\varphi^V_{2}}$ with a closed {\it compactly supported} differential $2$-form on ${X}$ representing a class in $H^2_c(X)$, see Theorem~\ref{KM90}). Our first result is 

\begin{theorem}\label{FM-boundaryexact}  (Theorem~\ref{globalexact})
The differential form $\theta_{\varphi^V_{2}}$ on $X$ extends to a form on $\overline{X}$, and the restriction $k^{\ast}$ of $\theta_{\varphi^V_{2}}$ to $\partial \overline{X}$ gives an {\it exact} differential form on $\partial\overline{X}$. Moreover, there exists a theta series $\theta_{\phi_1^W}$ for a space $W$ of signature $(1,1)$  of weight $2$ with values in the $1$-forms on $\partial \overline{X}$ such that $\theta_{\phi_1^W}$ is a primitive for $k^{\ast} \theta_{\varphi^V_{2}}$:
\[
d (\theta_{\phi_1^W}) = k^{\ast} \theta_{\varphi^V_{2}}.
\]
\end{theorem}

Considering the mapping cone for the inclusion $k: \partial \overline{X} \hookrightarrow \overline{X}$ (see Section~\ref{mappingconesection}) we then view the pair $[\theta_{\varphi^V_{2}}, \theta_{\phi_1^W}]$ as an element of the compactly supported cohomology $H^2_c(X)$. Explictly, let $C$ be a relative cycle in $\overline{X}$ representing a class in $H_2({X},\partial {X},\Z)$. Then the Kronecker pairing between $[\theta_{\varphi^V_{2}}, \theta_{\phi_1^W}]$ and $C$ is given by
\begin{equation}\label{mappingconelift}
\langle [\theta_{\varphi^V_{2}}, \theta_{\phi_1^W}], C \rangle = \int_C  \theta_{\varphi^V_{2}}
 - \int_{\partial C} \theta_{\phi_1^W}.
\end{equation}
In this way, we obtain an extension of the geometric theta lift which captures the non-compact situation.

To describe the geometric interpretation of this extension, we study the cycle $C_n$ at the boundary $\partial \overline{X}$ (Section~\ref{capped-cycles}). The intersection of $C_n$ with $\partial \overline{X}$ is a union of circles contained in the torus fibers of Sol. But rationally such circles are homologically trivial. Hence we can find a (suitably normalized) rational $2$-chain $A_n$ in $\partial \overline{X}$ whose boundary is the boundary of $C_n$ in $\partial \overline{X}$. ``Capping'' off $C_n$ by $A_n$, we obtain a {\it closed} cycle $C_n^c$ in $\overline{X}$ defining a class in $H_2({X},\Q)$. Our main result is the extension of \eqref{KM-id}:
\begin{theorem}\label{FMHZ-main} (Theorem~\ref{FM-main-th})
 Let $C$ be a relative cycle in $\overline{X}$. Then
\[
\langle [\theta_{\varphi^V_{2}}, \theta_{\phi_1^W}], C \rangle  = \sum_{n \geq 0} (C_n^c \cdot C)q^n
\]
is a holomorphic modular form of weight $2$ and is equal to the generating series of the  intersection numbers with the capped cycles $C^c_n$. (Similarly for the pairing with an arbitrary closed $2$-form on $\overline{X}$ representing a class in $H^2(X)$).
\end{theorem}

 Note that in view of \eqref{mappingconelift} the lift of classes of $H_2({X}, \partial {X})$ or $H^2(X)$ is the sum of two in general non-holomorphic modular forms (see below).

In \cite{FMres} we systematically study for $\Orth(p,q)$ the restriction of the classes $\theta_{\varphi^V_{q}}$ (also with non-trivial coefficients) to the Borel-Serre boundary. Whenever the restriction vanishes cohomologically, we can expect that a similar analysis to the one given in this paper will give analogous extensions of the geometric theta correspondence. In fact, aside from this paper we have at present managed to do this for several other cases, namely for modular curves with non-trivial coefficients \cite{FMspec} generalizing work of Shintani \cite{Shintani} and for Picard modular surfaces \cite{FM-Cogdell} generalizing work of Cogdell \cite{Cogdell}.

\subsubsection*{Linking numbers in $3$-manifolds of type Sol}

The theta series $\theta_{\phi_1^W}$ at the boundary is of independent interest and has geometric meaning in its own right. Recall that for two disjoint (rationally) homological trivial $1$-cycles $a$ and $b$ in a $3$-manifold $M$ we can define the {\it linking number} of $a$ and $b$ as the intersection number 
\[
\Lk(a,b) = A \cdot b
\]
of (rational) chains in $M$. Here $A$ is a $2$-chain in $M$ with boundary $a$. We show 

\begin{theorem}\label{FM-linking} (Theorem~\ref{xi'-integralP})
Let $c$ be homologically trivial $1$-cycle in $\partial \overline{X}$ which
is disjoint from the torus fibers containing components of $\partial C_n$. Then the holomorphic part of the weight $2$ non-holomorphic modular form $\int_c \theta_{\phi_1^W}$ is given by the generating series of the linking numbers $\sum_{n>0}\Lk(\partial C_n,c) q^n$.
\end{theorem}

We also give a simple formula in Theorem~\ref{linkSol} for the linking number of two circles contained in the fiber of a $3$-manifold $M$ of type Sol in terms of the glueing homeomorphism for the bundle.

One can reformulate the previous theorem stating that $\sum_{n>0}\Lk(\partial C_n,c) q^n$ is a ``mixed Mock modular form'' of weight $2$; it is the product of a Mock modular form of weight $3/2$ with a unary theta series. Such forms, which originate with the famous Ramanujan Mock theta functions, have recently generated great interest. 

Theorem~\ref{FM-linking} (and its analogues for the Borel-Serre boundary of modular curves with non-trivial coefficients and Picard modular surfaces) suggest that there is a more general connection between modular forms and linking numbers of nilmanifold subbundles over special cycles in nilmanifold  bundles over locally symmetric spaces.

\subsubsection*{Relation to the work of Hirzebruch and Zagier}

In their seminal paper \cite{HZ}, Hirzebruch-Zagier provided a map from the second homology of the smooth compactification of certain Hilbert modular surfaces $j:X \hookrightarrow \tilde{X}$ to modular forms. They introduced the Hirzebruch-Zagier curves $T_n$ in $X$, which are given by the closure of the cycles $C_n$ in $\tilde{X}$. They then defined ``truncated'' cycles $T_n^c$ as the
projections of $T_n$ orthogonal to the subspace of $H_2(\tilde{X},\Q)$ spanned by the 
compactifying divisors of $\tilde{X}$. The principal result of \cite{HZ} was that $\sum_{n \geq 0} [T_n^c] q^n$ defines a holomorphic modular form of weight $2$ with values in $H_2(\tilde{X},\Q)$.
We show $j_{\ast} C_n^c = T_n^c$ (Proposition~\ref{CnTn}), and hence the Hirzebruch-Zagier theorem follows easily from Theorem~\ref{FMHZ-main} above, see Theorem~\ref{HZTheorem}.

The main work in \cite{HZ} was to show that the generating function
\[\vspace{-.1cm}
F(\tau) = \sum_{n=0}^{\infty} (T_n^c \cdot T_m) q^n \vspace{-.1cm}
\]
for the intersection numbers in $\tilde{X}$ of $T^c_n$ with a fixed $T_m$ is a modular form of weight $2$. The Hirzebruch-Zagier proof of the modularity of $F$ was a remarkable synthesis of algebraic geometry, combinatorics, and modular forms. They explicitly computed the intersection numbers $T_n^c \cdot T_m$ as the sum of two terms, $T_n^c \cdot T_m = (T_n \cdot T_m)_X  + ({T}_n \cdot {T}_m)_{\infty}$, where $(T_n \cdot T_m)_X $ is the geometric intersection number of $T_n$ and $T_m$ in the interior of $X$ and $({T}_n \cdot {T}_m)_{\infty}$ which they called the ``contribution from infinity''. They then proved both generating functions $\sum_{n=0}^{\infty} (T_n \cdot T_m)_X  q^n$ and $\sum_{n=0}^{\infty}  (T_n \cdot T_m)_{\infty} q^n$ are the holomorphic parts of two non-holomorphic forms $F_X$ and $F_{\infty}$ with the {\it same} non-holomorphic part (with opposite signs). Hence combining these two forms gives $F(\tau)$.

We recover this feature of the original Hirzebruch-Zagier proof via \eqref{mappingconelift} with $C=C_m$. The first term on the right hand side of \eqref{mappingconelift} was studied in the thesis of the first author of this paper \cite{FCompo} and gives the interior intersections $(T_n \cdot T_m)_X$ encoded in $F_X$. So via Theorem~\ref{FMHZ-main} the second term on the right hand side of \eqref{mappingconelift} must match the boundary contribution $F_{\infty}$ in \cite{HZ}, that is, we obtain
\begin{theorem}
\[
({T}_n \cdot {T}_m)_{\infty} = \Lk( \partial C_n, \partial C_m).
\]
\end{theorem}
Hence we give an interpretation for the boundary contribution in \cite{HZ} in terms of linking numbers in $\partial \overline{X}$. In fact, the construction of $\theta_{\phi_1^W}$ owes a great deal to Section~2.3 in \cite{HZ}, where a scalar-valued version of $\theta_{\phi_1^W}$ is introduced, see also Example~\ref{HZbeta}. Using Theorem~\ref{LinkCnCm} one can also make the connection between our linking numbers and the formulas of the boundary contribution in \cite{HZ} explicit. 

To summarize, we start with the difference of theta integrals \eqref{mappingconelift} (which we know a priori is a holomorphic modular form), then by functorial differential topological computations we relate its Fourier coefficients to intersection/linking numbers, and by direct computation of the integrals involved we obtain the explicit formulas of Hirzebruch-Zagier and a ``closed form'' for their generating function.

Note that Bruinier \cite{B-123} and Oda \cite{Oda} use related theta series to consider \cite{HZ}, but their overall approach is different.

\subsubsection*{Currents}

One of the key properties of the cocycle $\varphi^V_2$ is that the $n$-th Fourier coefficients of $\theta_{\varphi^V_{2}}$ represents the Poincar\' e dual class for the cycle $C_n$. Kudla-Millson establish this by showing that $\varphi^V_2$ gives rise to a Thom form for the normal bundle of each of the components of $C_n$. To prove our main result, Theorem~\ref{FMHZ-main}, we follow a different approach using currents which is implicit in \cite{BFDuke} and is closely related to the Green's function $\Xi(n)$ for the divisors $C_n$ constructed by Kudla \cite{KAnn97,KBforms}. This function plays an important role in the Kudla program (see eg \cite{Kmsri}) which considers the analogous generating series for the special cycles in arithmetic geometry. In the non-compact situation however, one needs to modify  $\Xi(n)$ to obtain a Green's function for the cycle $T_n^c$ in $\tilde{X}$. Discussions with U. K\"uhn suggest that the constructions in this paper indeed give rise to such a modification of $\Xi(n)$, see Remark~\ref{Kudla-modification}.

\vskip.5cm

We would like to thank Rolf Berndt, Jan Bruinier, Jose Burgos, Misha Kapovich, and Ulf K\"uhn for fruitful and extensive discussions on the constructions and results of this paper. As always it is a great pleasure to thank Steve Kudla for his interest and encouragement. Each of us began the work of relating theta lifts and special cycles with him.

We dedicate this paper to the memory of Gretchen Taylor Millson, beloved wife of the second author.

\section{The Hilbert modular surface and  its Borel-Serre compactification}

\subsection{The symmetric space and its arithmetic quotient}

\subsubsection{The orthogonal group and its symmetric space}

Let $V$ be a rational vector space of dimension $4$ with a
non-degenerate symmetric bilinear form  $(\,,\,)$ of signature
$(2,2)$.
We let $\underline{G} = \SO(V)$, viewed as an algebraic group over $\Q$.
We let $G=\underline{G}_0(\R) \simeq \SO_0(2,2)$ be the connected component of the identity of the real points of $\underline{G}$.
It is most convenient to identify the associated symmetric space $D= D_V$ with the
space of negative $2$-planes in $V(\R)$ on which the bilinear form $(\,,\,)$ is
negative definite:
\[
D = \{z \subset V_{\R} ; \;\text{$\dim z =2$ and $(\,,\,)|_z < 0$}
\}.
\]
We pick an orthogonal basis $\{e_1,e_2,e_3,e_4\}$ of $V_{\R}$ with $(e_1,e_1)=(e_2,e_2)=1$ and $(e_3,e_3)=(e_4,e_4)=-1$. We denote the coordinates of a vector $x$ with respect to this basis by $x_i$. We pick as base point of $D$ the plane $z_0=[e_3,e_4]$ spanned by $e_3$ and $e_4$, and we let $K \simeq \SO(2)\times \SO(2)$ be the maximal compact subgroup of $G$ stabilizing $z_0$. Thus $D \simeq G/K$. Of course, $D \simeq \H \times \h$, the product of two upper half planes.

We let  $\underline{P}$ be a rational parabolic subgroup stabilizing a rational isotropic line $\ell$ and define $P= \underline{P}_0(\R)$ as before. We let $\underline{N}$ be its unipotent subgroup and $N = \underline{N}(\R)$. 
We let $u =(e_1+e_4)/\sqrt{2}$ and $u' =(e_1-e_4)/\sqrt{2}$ be two isotropic vectors so that $(u,u')=1$. We assume that $u,u'$ are defined over $\Q$ and obtain a rational Witt decomposition 
\[
V = \ell \oplus W \oplus \ell'
\]
with $\ell = \Q u$, $\ell'=\Q u'$, and a subspace $W = \ell^{\perp} \cap {\ell'}^{\perp}$ such that $W_{\R} = \Span_{\R}(e_2,e_3)$. The choice of $u'$ gives a Levi splitting of $\underline{P}$, and we write
\[
P =NAM
\]
for the Langlands decomposition. Here, with respect to the basis $u,e_2,e_3,u'$, we have  
\begin{align*}
{N} &=
 \left\{ n(w)   = \left( 
\begin{smallmatrix}
1&(\cdot,w)& -(w,w)/2 \\
 &   1_W   & -w  \\
 &         & 1  \\
\end{smallmatrix} \right)
; \; w \in W_{\R} \right\}, \\ 
{A} &= \left\{ a(t) = 
\left( \begin{smallmatrix}
t& &  \\
 &   1_W   &  \\
 &         & t^{-1}\\
\end{smallmatrix} \right); \, t \in \R_+
 \right\}, \\ 
{M} & =
\left\{ m(s)  =
\left( \begin{smallmatrix}
1& & \\
 &    \begin{smallmatrix} \cosh(s) & \sinh(s) \\ \sinh(s) & \cosh(s)
 \end{smallmatrix}    & \\\
 &         & 1  \\
\end{smallmatrix} \right)
; \; s \in \R \right\}.
\end{align*}
Note $N \simeq W_{\R}$. We obtain coordinates for $D$ by $z=z(t,s,w)$ where $z$ is the negative two-plane in $V_{\R}$ with $z=[n(w)a(t)m(s)e_3,n(w)a(t)m(s)e_4]$.

\subsubsection{Arithmetic Quotient}

We let $L$ be an even lattice in $V$ of level $N$, that is $L \subseteq L^{\#}$, the dual lattice, $(x,x) \in 2 \Z$ for $x \in L$, and $q(L^{\#}) \Z = \tfrac1{N}\Z$. We fix $h \in L^{\#}$ and let $\Gamma \subseteq \Stab{L}$ be a subgroup of finite index of the stabilizer of $\mathcal{L}:=L+h$ in $G$. For each isotropic line $\ell =\Q u$, we assume that $u$ is primitive in the lattice $L$ in $V$. We will throughout assume that the $\Q$-rank of $\underline{G}$ is $1$, that is, $V$ splits exactly one hyperbolic plane over $\Q$. Then we define the Hilbert modular surface
\[
X = \G \back D.
\]

\begin{example}\label{HZex}

An important example is the following. Let $d>0$ be the discriminant of the real quadratic field $K = \Q(\sqrt{d})$ over $\Q$, $\mathcal{O}_K$ its ring of integers. We denote by $x \mapsto x'$
the Galois involution on $K$. We let $V \subset M_2(K)$ be the space
of skew-hermitian matrices in $M_2(K)$, i.e., which satisfy $^tx' =-x$. Then the determinant on $M_2(K)$ gives $V$ the structure of a non-degenerate rational quadratic space of signature $(2,2)$ and $\Q$-rank $1$. We define the integral skew-hermitian matrices by 
\begin{equation*}
L = \left\{ x = \left( \begin{smallmatrix} a\sqrt{d}&\lambda\\-\lambda'&b\sqrt{d}
  \end{smallmatrix}  \right) \; : \; a,b \in \Z, \; \lambda  \in
  \mathcal{O}_K \right\}.
\end{equation*}
 Then $L$ is a lattice of level $d$. We embed $\SL_2(K)$ into $\SL_2{\R} \times \SL_2(\R)$ by $g \mapsto (g,g')$ so that $\SL_2(\mathcal{O}_K)$ acts on $L$ by $\g.x = \g x{^t\g'}$ as isometries. Hirzebruch and Zagier actually considered this case for $d \equiv 1 \pmod{4}$ a prime.
\end{example}

The quotient space $X$ is in general an oriented uniformizable orbifold with isolated singularities.  We will treat $X$ as a manifold - we will use Stokes' Theorem and Poincar\'e duality over $\Q$ on $X$. This is justified because in each instance we can pass to a finite normal cover $Y$ of $X$ with $Y$ a manifold. Hence, the formulas we want hold on $Y$. 
We then then go back to the quotient by taking invariants or summing over the group $\Phi$ of covering transformations. The point is that the de Rham complex of $X$ is the algebra of $\Phi$-invariants in the one of $Y$ and the {\it rational} homology (cohomology) groups of $X$ are the groups of $\Phi$-coinvariants (invariants) of those
 of $Y$.

\subsection{Compactifications}

\subsubsection{Admissible Levi decompositions of $P$}

We let $\G_P = \G \cap P$ and $\G_N = \G_P \cap N$. Then the quotient $\G_P/\G_N$ is a non-trivial arithmetic subgroup of $\underline{P}/\underline{N}$ and lies inside the connected component of the identity of the real points of $\underline{P}/\underline{N}$. Furthermore, $\G_P/\G_N$ acts as isometries of spinor norm $1$ on the anisotropic quadratic space $\ell^{\perp}/\ell$  of signature $(1,1)$. Hence $\G_P/\G_N \simeq \Z$ is infinite cyclic. Therefore the exact sequence 
\begin{equation*} \label{exactsequence}
1 \to \G_N \to \G_P \to \G_P/\G_N \to 1
\end{equation*}
splits. We fix $g \in \G_P$ such that its image $\bar{g}$ generates $\G_P/\G_N$. Then $g$ defines a Levi subgroup $M$. In fact, the element $g$ generates $\G_M :=\G_P \cap M$. Hence
\[
\G_P = \G_M \ltimes \G_N.
\]
We will say a Levi decomposition $P = NAM$ is admissible if $
\G_P = (M \cap \G_P) \ltimes \G_N$. In the following we assume that we have picked an admissible Levi decomposition for each rational parabolic.

\subsubsection{Borel-Serre compactification}\label{BScomp}

We let $\overline{D}$ be the (rational) Borel-Serre enlargement of $D$, see \cite{BorelSerre} or \cite{BJ}, III.9. For any parabolic $\underline{P}$ as before with admissible Levi decomposition $P=NAM$, we define the boundary component
\begin{equation*}
e({P}) = MN \simeq D_W \times W.
\end{equation*}
Here $D_W \simeq M \simeq \R$ is the symmetric space associated to the orthogonal group of $W$. Then $\overline{D}$ is given by 
\begin{equation*}
\overline{D} = D \cup \coprod_{\underline{P}} e({P}),
\end{equation*}
where $\underline{P}$ varies over all rational parabolics. The action of $\G$ on $D$ extends to $\overline{D}$ in a natural way, and we let 
\begin{equation*}
\overline{X}:= \G \back \overline{D}
\end{equation*}
be the Borel-Serre compactification of $X = \G \back D$. This makes $\overline{X}$ a manifold with boundary such that
\[
\partial \overline{X} = \coprod_{[\underline{P}]} e'({P}),
\]
where for each cusp, the corresponding boundary component is given by
\[
e'(P) = \G_P \back e(P).
\]
Here $[\underline{P}]$ runs over all $\G$-conjugacy classes. The space $X_W := \G_M \back D_W$ is a circle. Hence $e'(P)$ is a torus bundle over the circle, where the torus $T^2$ is given by $\G_N \back N$. That is, $
e'(P) = X_W \times T^2$, and we have the natural map $\kappa: e'(P) \to X_W$. We have a natural product neighborhood of $e(P)$ in $\overline{D}$ and hence for $e'(P)$ in $\overline{X}$ given by $[(T,\infty] \times e'(P)]$ for $T$ sufficiently large given by $z(t,s,w)$ with $t>T$. We let $i: X \hookrightarrow \overline{X}$ and $i_P: e'(P)\hookrightarrow \overline{X} $ be the natural inclusions.

It is one of the fundamental properties of the Borel-Serre compactification $\overline{X}$ that it is homotopic equivalent to $X$ itself. Hence their (co)homology groups coincide.

\subsubsection{Hirzebruch's smooth compactification}\label{H-compact}

We let $X'$ be the Baily-Borel compactifciation of $X$, which is obtained by collapsing in $\overline{X}$ each boundary component $e'(P)$ to a single point or topologically by taking a cone on each component of the Borel-Serre boundary. It is well known that $X'$ is a projective algebraic variety. We let $\tilde{X}$ be Hirzebruch's smooth resolution of the cusp singularities and $\pi:\tilde{X} \to X'$ be the natural map collapsing the compactifying divisors for each cusp. We let $j:X \hookrightarrow \tilde{X}$ be the natural embedding. Note that the Borel-Serre boundary separates $\tilde{X}$ into two pieces, the (connected) inside $X^{in}$, which is isomorphic to $X$ and the (disconnected) outside $X^{out}$, which for each cusp is a neighborhood of the compactifying divisors. Note that we can view $e'(P)$ as lying in both $X^{in}$ and $X^{out}$ since the intersection $X^{in} \cap X^{out}$ is equal to $ \coprod_{\underline{P}} e({P})$.

\section{(Co)homology}

In this section we describe the relationship between the (co)homology of the various
compactifications. 

\subsection{The homology of the boundary components}\label{boundaryhom}

Every element of $\Gamma_N =\pi_1(T^2)$ is a rational multiple of a commutator in $\Gamma_P$ and accordingly the image of $H_1(T^2,\Q)$ in $H_1(e'(P),\Q)$ is trivial.
Let $a_P \in H_1(e'(P),\Z)$ be the class of the identity section of $\kappa:e'(P) \to X_W$ and $b_P \in H_2(e'(P),\Z)$ be the class of the torus fiber of $\kappa$. It is clear that the intersection number of $a_P$ and $b_P$ is $1$ (up to sign) whence $a_P$ and $b_P$ are nontrivial primitive classes. Furthermore, $a_P$ generates $H_1(e'(P),\Q)$ and $ H_2(e'(P),\Z)  \cong \Z$, generated by $b_P$. So

\begin{lemma}\label{ePhomology}

\begin{enumerate}
\item[(i)] The first rational homology group of $e'(P)$ is generated by $a_P$.
\item[(ii)] The second homology group of $e'(P)$ is generated by $b_P$.
\end{enumerate}
\end{lemma}
\begin{remark} To compute the homology over $\Z$ one has only to use the Wang sequence for a fiber bundle over a circle, see \cite{Milnor}, page 67.
\end{remark}

Let $\Omega_P$ be the unique $P$-invariant $2$-form on $e'(P)$ such that
\begin{equation}\label{areaform}
\int_{b_P} \Omega_P = 1.
\end{equation}
Since $b_P$ is the image of the fundamental class of $T^2$ inside $H_2(e'(P),\Z)$, we see that that the restriction of $\Omega_P$ to $T^2$ lifts to the area form on $W_{\R} \simeq N$ normalized such that $T^2=\G_N \back N$ has area $1$.

\subsection{Homology and cohomology of $X$ and $\tilde{X}$}

Accordingly to the discussion in Section~\ref{H-compact} we have the Mayer-Vietoris sequence
\[
0 \to \oplus_P H_2(e'(P)) \to H_2(X) \oplus (\oplus_P S_P)  \to H_2(\tilde{X}) \to 0.
\]
Here $S_{P}$ denotes the span of the classes defined by compactifying divisors at the cusp associated to $P$. The zero on the left comes from $H_3(\tilde{X}) =0$ and the zero on the right comes from the fact that for each $P$ the class $a_P$ injects into $H_1(X^{out})$, see \cite{vGeer}, II.3. Since the generator $b_P$ has trivial intersection with each of the compactifying divisors, $b_P$ bounds on the outside so a fortiori it bounds in $\tilde{X}$. Thus the above short exact sequence is the sum
of the two short exact sequences $\oplus_P H_2(e'(P)) \to H_2(X) \to j_{\ast} H_2(X)$ and
$ 0 \to \oplus_P S_P \to \oplus_P S_P $.  By adding the third terms of the two sequences and equating them to $H_2(\tilde{X})$  we obtain the orthogonal splittings (for the intersection pairing) - see also \cite{vGeer}, p.123,
\begin{align*}\label{vdG2}
H_2(\tilde{X}) = j_{\ast} H_2(X) \oplus \left( \oplus_{[P]} S_{P} \right), \qquad \qquad 
H^2(\tilde{X})  = j_{\#} H_c^2(X) \oplus \left(\oplus_{[P]} S^{\vee}_{P}\right).
\end{align*}
Here $j_{\#}$ is the push-forward map. Furthermore, the pairings on each summand are non-degenerate. Considering $\oplus_P H_2(e'(P)) \to H_2(X) \to j_{\ast} H_2(X)$ we also obtain
\begin{proposition}\label{intersectionhom}
$H_2(\partial \overline{X})$ is the kernel of $j_{\ast}$ so that
\[
j_{\ast} H_2(X) \simeq H_2(X)/ \sum_{[P]} H_2(e'(P)).
\]
\end{proposition}

\subsection{Compactly supported cohomology and the cohomology of the mapping cone}\label{mappingconesection}

We briefly review the mapping-cone-complex realization of the cohomology of compact supports of $X$. For a more detailed discussion, see \cite{FMspec}, section~5. 

We let $A_c^{\bullet}(X)$ be the complex of compactly supported differential forms on $X$ which gives rise to $H_c^{\bullet}(X)$, the cohomology of compact supports. We now represent the compactly-supported cohomology of $X$ by the cohomology of the mapping cone $C^{\bullet}$ of $i^*$, see \cite{Weibel}, p.19, where as before $i: X \hookrightarrow \overline{X}$. However, we will change the sign of the differential on $C^{\bullet}$ and shift the grading down by one. Thus we have
$$C^i =\{ (a,b), a \in A^i (\overline{X}), b \in A^{i-1}(\partial \overline{X})\}$$
with $d(a,b) = (da, i^*a - db)$.
If $(a,b)$ is a cocycle in $C^{\bullet}$ we will use $[[a,b]]$ to denote its cohomology class. We have
\begin{proposition} \label{quasiiso}
The cochain map $A_c^{\bullet}(X) \to C^{\bullet}$ given by $c \mapsto (c,0)$ is a quasi-isomorphism.
\end{proposition}

We now give a cochain map from $C^{\bullet}$ to $ A_c^{\bullet}(X)$ which induces the inverse to the above isomorphism. We let $V$ be a product neighborhood of $\partial \overline{X}$ as in Section~\ref{BScomp}, and we let $\pi:V \to \partial \overline{X}$ be the projection. If $b$ is a form on $\partial \overline{X}$ we obtain a form $\pi^{\ast} b$ on V. Let $f$ be a smooth function of the geodesic flow coordinate $t$ which is $1$ near $t=\infty$ and zero for $t \leq T$
for some sufficiently large $T$. We may regard $f$ as a function on $V$ by making it constant on the $\partial \overline{X}$ factor. We extend $f$ to all of $ \overline{X}$ by making it zero off of $V$. Let $(a,b)$ be a cocycle in $C^i$. Then there exist a compactly supported closed form $\alpha $ and a form $\mu$ which vanishes on $\partial \overline{X}$ such that
\[
a - d(f \pi^{\ast}b ) = \alpha + d\mu.
\]
We define the cohomology class $[a,b]$ in the compactly supported  cohomology $H^i_c(X)$ to be the class of $\alpha$, and the assignment $[[a,b]] \mapsto [a,b]$ gives the desired inverse. From this we obtain the  following integral formulas for the Kronecker pairings with $[a,b]$. 

\begin{lemma}\label{integralformula}
Let $\eta$ be a closed form on $\overline{X}$ and $C$ a relative cycle in $\overline{X}$ of appropriate degree. Then 
$$
\langle[a, b], [\eta]]\rangle 
= \int_{\overline{X}}a\wedge \eta - \int_{\partial \overline{X}} b \wedge i^*\eta, \ \text{and} \ \  
\langle [a,b],C \rangle =  \int_{C}a - \int_{\partial C} b.$$
\end{lemma}

\section{Capped special cycles and linking numbers in Sol}\label{capped-cycles}

For $x \in V$ such that $(x,x)>0$, we define 
\[
D_x =\{ z \in D; \, z \perp x \}.
\]
Then $D_x$ is an embedded upper half plane in $D$. We let $\G_x \subset \G$ be the stabilizer of $x$ and define the special or Hirzebruch-Zagier cycle by 
\[
C_x = \G_x \back D_x, 
\]
and by slight abuse identify $C_x$ with its image in $X$. These are modular or Shimura curves. For positive $n \in \Q$, we write $\calL_n = \{ x \in \mathcal{L}; \, \tfrac12(x,x)= n\}$. Then the composite cycles $C_n$ are given by
\[
C_n= \sum_{x \in \G \back \calL_n} C_x.
\]
Since the divisors define in general relative cycles, we take the sum in $H_2(X,\partial X,\Z)$.

\subsection{The closure of special cycles in the Borel-Serre boundary and the capped cycle $C_x^c$}

We now study the closure of $C_x$ in $\partial \overline{X}$, which is the same as the intersection of $\overline{C}_x$ or $\partial C_x$ with the union of the hypersurfaces $e'(P)$. A straightforward calculation gives 

\begin{proposition}\label{boundaryofC}
If $(x,u) \neq 0$ then there exists a neighborhood $U_{\infty}$ of $e(P)$ such that
$$
D_x \cap U_{\infty} = \emptyset.
$$
If $(x,u) = 0$, then $\overline{D}_x \cap e(P)$ is contained in the fiber of $p$ over $s(x)$, where $s(x)$ is the unique element of $\R$ satifying
\[
(x, m(s(x)) e_3) = 0.
\]
At $s(x)$ the intersection $\overline{D}_x \cap e(P)$ is the affine line in $W$ given by 
\[
\{ w \in W: (x,w) = (u',x)\}.
\]
\end{proposition}

We define $c_x \subset \partial C_x$ to be the closed geodesic in the fiber over $s(x)$ which is the image of $\overline{D}_x \cap e(P)$ under the covering $e(P) \to e'(P)$. We have

\begin{proposition}\label{TnBS}
\begin{enumerate}
\item[(i)] The $1$-cycle $\partial C_x$ is a finite union of circles. 
\item[(ii)] At a cusp associated to $P$, each circle is contained in a fiber of the map $\kappa: e'(P) \to X_W$ and hence is a rational boundary (by Lemma~\ref{ePhomology}). 
\item[(iii)] Two boundary circles $c_x$ and $c_y$ are parallel if they are contained in the same fiber. In particular, 
$
c_x \cap c_y \neq \emptyset \iff c_x= c_y.
$
\end{enumerate}
\end{proposition}

We now describe the intersection of $\overline{C}_n$ or $\partial C_n$ with $e'(P)$. For $\calL_V=\calL = L +h$ we can write
\begin{equation*}
{\calL}_W= {\calL}_{W_P}= (\calL \cap u^{\perp}) / (\calL \cap u) \simeq \coprod_k \left(L_{W,k} +
h_{W,k}\right)
\end{equation*}
for some lattices $L_{W,k} \subset W$ and vectors $h_{W,k} \in L^{\#}_{W,k}$.

Via the isomorphism $W \simeq N$, we can identify $\G_N = N \cap \G$ with a lattice $\Lambda_W$ in $W$. Since  $u$ is primitive in $L $ and $n(w) x= x + (w,x)u$ for a vector $x \in u^{\perp}$ we see that ${\calL}_W$ is contained in the dual lattice of $\Lambda_W$. 

\begin{lemma}\label{LemmaB}
The intersection $\partial C_n \cap e'(P)$ is given by 
\[
 \coprod_{ \substack{x\in \G_M \back \mathcal{L}_W \\ (x,x)=2n}} \coprod_{0 \leq k < \min'_{\la \in \Lambda_W}   |(\la,x)|} c_{x+ku}.
 \] 
Here $\min'$ denotes that we take the minimum over all nonzero values of $ |(\la,x)|$.
\end{lemma}

\begin{proof}

We will first prove $\partial C_{n,P} := \partial C_n \cap e'(P)$ is a disjoint union
\begin{equation} \label{union}
\partial C_{n,P} = \coprod_{y \in \G_P \back \calL_{n,u}} c_y,
\end{equation}
where $\calL_{n,u} = \{ x \in \calL \cap u^{\perp};\, (x,x)=2n\}$. Indeed, first note that by Proposition~\ref{boundaryofC} only vectors in $\calL_{n,u}$ can contribute to $\partial C_{n,P}$. 
The action of $\Gamma$ on $V$ induces an equivalence relation $\sim_{\Gamma}$ on the set $\G_p \back \calL_{n,u} \subset V$ which is consequently a
union of equivalence classes $[x_i]= [x_i]_P, 1 \leq i \leq k$.  We may accordingly organize the union $R$ on the right-hand side of \eqref{union} as $R =  \coprod _{i=1}^k  \coprod_{ y \in [x_i]} c_y.$
But it is clear that $(\partial C_{x_i})_P = \coprod_{ y \in [x_i]} c_y$ and hence we have the equality of $1$-cycles in $e'(P)$
and $\partial X$
\begin{equation}\label{boundaryofspecialcycle}
(\partial C_{x_i})_P = \sum_{ y \in [x_i]_P} c_y \qquad  \text{and} \qquad  \partial C_{x_i}= \sum_{[P]} \sum_{ y \in [x_i]_P} c_y,
\end{equation}
since an element $y  \in [x_i]$ gives rise to the lift $D_y$ of $C_{x_i}$ to $D$ that intersects
$e(P)$ and this intersection projects to $c_y$.  Thus we may rewrite the right-hand side of\eqref{union} as
$R= \coprod_{  \sim_{\Gamma} \back \calL_{n,u}} \partial C_{x_i}.$
But it is clear that this latter union is $\partial C_{n,P}$ and \eqref{union} follows. Finally, we easily see that $\coprod_{ \substack{x\in \G_M \back \mathcal{L}_W \\ (x,x)=2n}} \coprod_{0 \leq k < \min'_{\la \in \Lambda_W}   |(\la,x)|} x+ku$ is a complete set of representatives of $\G_P$-equivalence classes in $\calL_{n,u}$. These give the circles $c_{x+ku}$ above. 

\end{proof}

\begin{proposition}\label{rat-cap}
Let $x \in \mathcal{L}_{n,u}$ with $n>0$. Then there exists a rational $2$-chain $a_x$ in $e'(P)$ such that
\begin{enumerate}
\item $\partial a_x = c_x$
\item $\int_{a_x} \Omega_P = 0$, here $\Omega_P$ is the area form for the fibers (see \eqref{areaform}) 
\end{enumerate}

\end{proposition}

\begin{proof}

Except for the rationality of the cap this follows immediately from Proposition~\ref{TnBS}. The problem is to find a cap $a_x$ such that $\int_{a_x} \Omega_P \in \Q$. We will prove this in Section~\ref{rat-cap11} below. 
\end{proof}

We will define $(A_x)_P$ by $(A_x)_P = \sum_{y \in [x]} a_x$. Then sum over the components $e'(P)$ to obtain $A_x$ a rational $2$-chain 
in $\partial X$. Then we have (noting that $(\partial C_x)_P = \sum_{y \in [x]} c_y$)
\[
\partial A_x = \partial C_x.
\]

\begin{definition}
We define the rational absolute $2$-cycle in $\overline{X}$ by 
\[
C_x^c = C_x \cup (-A_x)
\]
with the $2$-chain $A_x$  in $\partial \overline{X}$ as in Proposition~\ref{rat-cap}. In particular, $C_x^c$ defines a class in $H_2(\overline{X}) = H_2(X)$. In the same way we obtain $C_n^c$. 
\end{definition}

\subsection{The closure of the special cycles in $\tilde{X}$ and the cycle $T_n^c$}

Following Hirzebruch-Zagier we let $T_n$ be the cycle in $\tilde{X}$ given by the closure of the cycle $C_n$ in $\tilde{X}$. Hence $T_n$ defines a class in $H_2(\tilde{X})$.

\begin{definition}
Consider the decomposition $H_2(\tilde{X}) = j_{\ast} H_2(X) \oplus \left( \oplus_{[P]} S_{p
} \right)$, which is orthogonal with respect to the intersection pairing on $\tilde{X}$. We let $T_n^c$ be the image of $T_n$ under orthogonal projection onto the summand $j_{\ast} H_2(X)$.
\end{definition}

\begin{proposition}\label{CnTn}
We have
\[
j_{\ast} C_n^c = T_n^c.
\]
\end{proposition}

\begin{proof}
For simplicity, we assume that $X$ has only one cusp. The $3$-manifold $e'(P)$ separates $T_n$
and  we can write $T_n = T_n \cap X^{in} + T_n \cap X^{out}$ as (appropriately oriented)  $2$-chains  in $\tilde{X}$. It is obvious that we have $j_{\ast} \overline{C}_n = T_n \cap X^{in}$ as $2$-chains. We write $B_n = T_n \cap X^{out}$. We have $\partial C_n = - \partial B_n$. Hence we can write $T_n = j_{\ast} C_n^c + B_n^c$, the sum of two  $2$-cycles in $\tilde{X}$. Here $B_n^c$ is obtained by `capping' $B_n$ in $e'(P)$ with the negative of the cap $A_n$ of $C_n^c$. 
Since  $j_*C_n^c$  is clearly orthogonal to $S_P$ (since it lies in $X^{in}$) and $B_n^c \in
S_P$ (since it lies in $X^{out}$)
the decomposition $T_n = j_*C_n^c + B_n^c$ is just the decomposition of
$T_n$ relative to the splitting $H_2(\tilde{X}) = j_*H_2(X)  \oplus S_P$.
Hence $T_n^c =  j_*C_n^c$, as claimed.
\end{proof}

\subsection{Rationality of the cap}\label{rat-cap11}
We will now prove Proposition \ref{rat-cap}.  In fact we will show that it holds for any circle $\alpha$ contained in a torus fiber of $e(P)$ and passing through a rational point. We would like to thank Misha Kapovich for simplifying  our original argument. The idea is to construct, for each component of $\partial C_x$, a $2$-chain $A$ with that component as boundary so that $A$ is a sum $P+ T  +\mathcal{M}(\gamma_0)$ of three simplicial $2$-chains in $M$. We then verify that the ``parallelogram'' $P$  and the ``triangle'' $T$  have rational area and the period of $\Omega$ over the ``monodromy chain'' $\mathcal{M}(\gamma_0)$ is zero. 

In what follows we will pass from pictures in the plane involving directed line segments, triangles and parallelograms to identities in the space of simplicial $1$-chains $C_1(T^2)$  on $T^2$. The principal behind this is that any $k$-dimensional subcomplex $S$ of a simplicial complex $Y$ which is the fundamental cycle of an oriented $k$-submanifold $|S|$ (possibly with boundary) of $Y$ corresponds in a {\it unique} way to a sum of oriented $k$-simplices in $C_k(Y)$.

In this subsection we will work with a general $3$-manifold $M$ with Sol geometry. Of course this includes all the manifolds $e'(P)$ that occur in this paper. Let $f \in \SL(2,\Z)$ be a hyperbolic element. We will then consider the $3$-manifold $M$ obtained from $ \R \times T^2$ (with the $2$-torus $T^2 = W/ \Z^2$) given by the relation
\begin{equation}\label{glueing}
(s,w) \sim (s+1,f(w)).  
\end{equation}
We let $\pi: \R\times T^2 \to M$ be the resulting infinite cyclic covering.

We now define notation we will use below. We will use Greek letters to denote closed geodesics on $T^2$, a subscript $c$  will indicate that the geodesic starts at the point $c$ on $T^2$. We will use the analogous notation for geodesic arcs on $W$.  We will use $[\alpha]$ 
to denote the corresponding homology class of a closed geodesic $\alpha$ on $T^2$. 
If $x$ and $y$ are points on $W$ we will use $\overline{xy}$ to denote the oriented line segment joining $x$ to $y$ and $\overrightarrow{xy}$ to denote the corresponding (free) vector
i.e. the equivalence class of $\overline{xy}$ modulo parallel translation. 
   
We first take care of the fact that $\alpha$ does not necessarily pass through the origin. For convenience we will assume $\alpha$ is in the fiber over the base-point $z(x)$ corresponding to $s=0$. Let $\alpha_0$ be the parallel translate of $\alpha$ to the origin. Then we can find a cylinder $P$, image of an oriented parallelogram $\widetilde{P}$ under the universal cover $W \to T^2$ with rational vertices, such that in $Z_1(T^2,\Q)$, the group of rational $1$-cycles, we have
\begin{equation}\label{firstrectangle}
\partial P = \alpha - \alpha_0.
\end{equation}
Since $\widetilde{P}$ has rational vertices we find $\int_{P} \Omega = \int_{\widetilde{P}} \Omega \in \Q$.

Now we take care of the harder part of finding $A$ as above. The key is the construction  of ``monodromy $2$-chains''.  For any closed geodesic $\gamma_0 \subset T^2$ starting at $0$ we define the monodromy $2$-chain  $\mathcal{M}(\gamma_0)$ to be the image of the cylinder $\gamma_0 \times [0,1] \subset T^2 \times \R$ in $M$.
The reader will verify using \eqref{glueing} that in $Z_1(T^2,\Q)$ we have 
\begin{equation} \label{boundaryofmon}
 \partial \mathcal{M}(\gamma_0) = f^{-1}(\gamma_0) -\gamma_0.
\end{equation}
Since $f$ preserves the origin, the geodesic $f^{-1}(\gamma_0)$ is also a closed geodesic starting at the origin. Since $f^{-1}$ is hyperbolic we have $|\tr(f^{-1})| >2$ and hence $\det(f^{-1} -I)=  det( I - f) = \tr(f) -2 \neq 0$. Put $N= \det(f^{-1} -I)$ and define $[\gamma_0] \in H_1(T^2,\Z)$ by  
\begin{equation}\label{invertmatrix}
f^{-1}([\gamma_0]) -[\gamma_0] = N[\alpha_0]. 
\end{equation}
Note that $[\gamma_0] = N \{(f^{-1} - I)^{-1} ([\alpha_0] \}$ is necessarily an integer homology class. Also note that is an equation in the first {\it homology}, it is not an equation in the group of $1$-cycles $Z_1(T^2,\Q)$. Since any homology class contains a unique closed geodesic starting at the origin we obtain a closed geodesic $\gamma_0 \in [\gamma_0]$  and a corresponding  monodromy $2$-chain $\mathcal{M}(\gamma_0)$ whence \eqref{boundaryofmon} holds in $Z_1(T^2,\Q)$. We now solve

\begin{problem}
Find an equation in $Z_1(T^2,\Z)$ which descends to  \eqref{invertmatrix}. 
\end{problem}

Let $h_1$ resp. $h_2$ denote the covering transformation of $\pi$ corresponding to the element $\alpha_0$ resp $\gamma_0$ of the fundamental group of $T^2$. Define $c_1$ and $c_2$ in $W$ by $c_1 = Nh_1(0)$ and $c_2=h_2(0)$. Define $d \in W$ by $d =f^{-1}(c_2)$ in $W$. Let $\widetilde{T}$ be the oriented triangle
with vertices $0,c_2,d$.  Then
 \[
 (i) \quad \pi(\overline{0c_1}) =N\alpha_0 \qquad (ii) \quad \pi(\overline{0c_2})= \gamma_0 \qquad (iii) \quad 
 \pi(\overline{0d}) = \pi(f^{-1}(\overline{0c_2})) = f^{-1}(\gamma_0).
 \]
 We now leave it to the reader to combine the homology equation \eqref{invertmatrix} and the three equations to show the equality of directed line segments
\begin{equation}\label{equalityofsegments}
h_2( \overline{0c_1}) = \overline{c_2d}.
\end{equation}
With this we can solve the problem. 
We see that if we consider $\widetilde{T}$ as an oriented $2$-simplex we have the following equality of one chains
$$
\partial \widetilde{T} = \overline{0c_2} + \overline{c_2d} - \overline{0d}.
$$
Let $T$ be the image of $\widetilde{T}$ under $\pi$. Take the direct image of the previous equation under $\pi$ and use equation \eqref{equalityofsegments} which implies that the second edge $\overline{c_2d}$ is equivalent under $h_2$ in the covering group to the directed line segment $\overline{0c_1}$ which maps to $N\alpha_0$. Hence $\overline{c_2d}$ also maps to $N\alpha_0$. We obtain  the following equation in $Z_1(T^2,\Z)$
\begin{equation}\label{triangle}
\partial T =  \gamma_0 + N \alpha_0 - f^{-1}(\gamma_0),
\end{equation}
and we have solved the above problem. Combining \eqref{boundaryofmon} and \eqref{triangle} we have
$$
\partial (\mathcal{M}(\gamma_0) + T ) = f^{-1}(\gamma_0) -\gamma_0 +\gamma_0 + \alpha_0 - f^{-1}(\gamma_0)= N\alpha_0.
$$
Combining this with \eqref{firstrectangle} and setting where $A_0 = \mathcal{M}(\gamma_0) +T$ we obtain
\begin{equation} \label{cap}
\partial (NP + A_0 ) = N \alpha,
\end{equation}
in $Z_1(M,\Z)$. Hence if we define $A$ to be the  {\it rational} chain $A = \frac{1}{N} (NP + A_0) = P + \frac{1}{N}T + \frac{1}{N} \mathcal{M}(\gamma_0)$ in $M$ we have the following equation in $Z_1(M,\Q)$:
$$
\partial A = \alpha.
$$
Finally, the integral of $\Omega$ over $A$ is rational. Indeed, the integral over $P$ is rational.  Since all vertices of $\widetilde{T}$  are integral the area of $\widetilde{T}$ is integral, the integral of $\Omega$ over $T$ is integral. Thus it suffices to observe that the restriction of $\Omega$ to $\mathcal{M}(c)$ is zero. With this we have completed the proof of Proposition \ref{rat-cap}.

\subsection{Linking numbers in Sol} \label{generaltheoryoflink}

In the introduction we defined the linking number of two two disjoint homologically trivial $1$-cycles $a$ and $b$ in a closed $3$-manifold $M$ as $\Lk(a,b) = \langle A,b \rangle$, where $A$ is any rational $2$-chain in $M$ with boundary $a$. Since $b$ defines a trivial homology class in $M$, the link is well-defined, ie, does not depend on the choice of $A$. 

We let $M$ be the Sol manifold as before realized as in Section~\ref{rat-cap11} via \eqref{glueing} and consider the case when $a$ and $b$ are two contained in two torus fibers. Then by the previous section they are homologically trivial. If $a$ and $b$ are contained in the same fiber we move $b$ to the right (i.e. in the direction of positive $s$)  to a nearby fiber. We take $a,b \in H_1(T^2,\Z)$, and in this section we are allowed to confuse $a$ and $b$ with their representatives in the lattice $\Z^2$ and the unique closed geodesic in $T^2$ passing through the origin that represents them. We will write for the image of $a$ and $b$ in $\R \times T^2$ and $M$ $a=a(0)=0 \times a$ and $b=b(\eps)= \eps \times b$. Our goal is to compute the linking number $Lk(a, b(\epsilon))$. By the explicit construction of the cap $A$ in Section~\ref{rat-cap11} we obtain

\begin{lemma}
 \[
 Lk(a,b(\epsilon)) = M(c) \cdot b(\epsilon) = c(\epsilon) \cdot b(\epsilon)= c \cdot b.
 \] 
Here $c$ is the rational one cycle obtained by solving $(f^{-1} - I) (c) =a$
and $M(c)$ is the (rational) monodromy $2$-chain associated to $c$ (see above) with boundary $\partial M(c) = (f^{-1} - I) (c) =a$. Here the first $\cdot$ is the intersection of 
chains in $M$, the next $\cdot$ is the intersection number of $1$-cycles in the fiber $\epsilon \times T^2$ and the last $\cdot$ is the intersection number of $1$-cycles in $0 \times T^2$. 
 \end{lemma}

Noting that this last intersection number coincides with the intersection number of the underlying homology classes which in term coincides with the symplectic form $\langle \cdot, \cdot \rangle$ on $H_1(T^2,\Q)$ we have found our desired formula for the linking number.
\begin{theorem}\label{linkSol}
$ Lk(a, b(\epsilon)) = \langle (f^{-1} - I)^{-1} (a), b \rangle.$
\end{theorem}

It is a remarkable fact that there is a simple formula involving only the action of the
glueing homeomorphism $f \in  \SL(2,\Z)$ on $H_1(T^2, \Z)$ for linking numbers for $1$-cycles contained in fiber tori $T^2$ of in Sol (unlike the case of linking numbers in $\R^3$). 

This immediately leads to an explicit formula for the numbers $Lk(\partial C_n, \partial C_m)$. Using Lemma~\ref{LemmaB} we obtain

\begin{theorem}\label{LinkCnCm} 
Let $g = (f^{-1} - I)^{-1}$. Then
 \[
 Lk( (\partial C_n)_P, (\partial C_m)_P) = \sum_{ \substack{x\in \G_M \back \mathcal{L}_W \\ (x,x)=2n}} \sum_{\substack{x'\in \G_M \back \mathcal{L}_W \\ (x,x)=2m}} (\min_{\lambda \in \Lambda_W}  {\hspace{-5pt}'}
 |(\lambda,x)|) (\min_{\mu \in \Lambda_W}  {\hspace{-5pt}'}
|(\mu,x')|) \langle g(Jx),Jx' \rangle. 
\]
Here $Jx$ is properly oriented primitive vector in $\Lambda_W$ such that $(Jx,x)=0$. 
 \end{theorem}

 \begin{example}\label{LinkCnCmex} 
We consider the integral skew Hermitian matrices in Example~\ref{HZex}. Let $u= \kzxz{\sqrt{p}}{0}{0}{0}$, so that $W = \{ \kzxz{0}{\la}{-\la'}{0};\; \la \in K \} \simeq K$. The symplectic form on $K$ is given by $\langle \la, \mu \rangle = \frac{1}{\sqrt{p}} (\la \mu' - \la'\mu)$. The action of the unipotent radical $N= \left\{ n(\la)= \kzxz{1}{\la}{0}{1} \right\}$ on a vector $\mu \in K$ is now slightly different, namely, $n(\la) \mu = \mu + \langle \la, \mu \rangle u$. Hence in these coordinates, $\partial C_{\mu}$ is given by the image of the line $\R \mu = \{\la \in K_\R; \; \langle \la, \mu \rangle =0 \}$, and $(\min'_{\lambda \in \mathcal{O}_K}  
 |\langle \la, \mu \rangle|)\mu$ is a primitive generator in $\mathcal{O}_K$ for that line. We let $\eps$ be a generator of $U_+$, the totally positive units in $\mathcal{O}_K$, and we assume that the glueing map $f$ is realized by multiplication with $\eps'$. For $d \equiv 1 \pmod{4}$ a prime and $m=1$, $C_1$ has only component arising from $x =1 \in K$ and $C_1 \simeq \SL_2(\Z) \back \h$. Then Theorem~\ref{LinkCnCm} becomes (the $\min'$-term is now wrt $\langle\,,\, \rangle$)
\[
 Lk( (\partial C_n)_P, (\partial C_1)_P) = 
 2 \sum_{ \substack{\mu \in U_+ \back \mathcal{O}_K\\ \mu\mu'=n, \mu \gg 0}} \left\langle \tfrac{\mu}{\eps-1}, 1 \right\rangle = 2\sum_{ \substack{\mu \in U_+ \back \mathcal{O}_K\\ \mu\mu'=n, \mu \gg0}}  = \frac{2}{\sqrt{p}}\frac{\mu+\mu'\eps}{\eps-1}.
\]
This is (twice) the ``boundary contribution'' in \cite{HZ}, Section~1.4, see also Section~\ref{special-lift-section}.
\end{example}

\section{Schwartz functions and forms}

Let $U$ be a non-degenerate rational quadratic space of  signature $(p,q)$ and even dimension $m$. We will later apply the following to $U=V$ and $U=W$. Changing notation from before, we let $G = \SO_0(U_{\R})$ with maximal compact subgroup $K$ and write $D=G/K$ for the associated symmetric space. We let $\calS(U_{\R})$ be the space of Schwartz functions on $U_{\R}$ on which $\SL_2(\R)$ acts via the Weil representation $\omega$.

\subsection{Extending certain Schwartz functions to functions of $\tau \in \h$ and $z\in D$}\label{conventions}

Let $\varphi \in \calS(U_{\R})$ be an eigenfunction under the maximal compact $\SO(2)$ of $\SL_2(\R)$ of
weight $r$. Define $g'_{\tau} \in \SL_2(\R)$
by $g'_{\tau} = \left(
\begin{smallmatrix}1&u\\0&1\end{smallmatrix} \right) \left(
\begin{smallmatrix}v^{1/2}&0\\0&v^{-1/2}\end{smallmatrix} \right)$.
Then we have $ \omega(g'_{\tau}) \varphi (x)= v^{m/4} \varphi(\sqrt{v}x) e^{\pi i (x,x)u}$.
Accordingly we define 
\begin{equation}\label{group-tau}
\varphi(x,\tau)  = v^{-r/2} \omega(g'_{\tau}) \varphi
(x)  = v^{-r/2+m/4} \varphi^0(\sqrt{v}x) e^{\pi i (x,x)\tau}.
\end{equation}
Here we have also defined $\varphi^0(x) = \varphi(x) e^{\pi (x,x)}$. Let $E$ be a $G$-module and let $g_z \in G$ be any element that carries the basepoint $z_0$ in $D$ to $z \in D$. Then define for $\varphi \in [\calS(U_{\R}) \otimes E]^K$, the $E$-valued $K$-invariant Schwartz functions on $U_{\R}$, the functions $\varphi(x,z)$ and $\varphi(x,\tau,z)$ for $x \in U, z \in D, \tau \in \mathbb{H}$ by
\[
\varphi(x,z) =g_z \varphi(g_z^{-1}x) \qquad  \text{and} \qquad  \varphi(x,\tau,z) = g_z\varphi(g_z^{-1}x,\tau).
\]
We will continue to use these notational conventions for other (not necessarily Schwartz) functions that arise in this paper.

\subsection{Schwartz forms for $V$}\label{V-forms}

Let $\mathfrak{g} $ be the Lie algebra of $G$ and $\mathfrak{g}= \mathfrak{k} \oplus \mathfrak{p}$ be the Cartan decomposition of $\mathfrak{g}$ associated to $K$.  We identify 
 $\mathfrak{g} \simeq \wwedge{2} V_{\R}$ as usual via $
(v_1 \wedge v_2)(v) = (v_1,v)v_2 - (v_2,v)v_1$. We write $X_{ij} = e_i \wedge e_j \in \mathfrak{g}$ and note that $\mathfrak{p}$ is spanned by $X_{ij}$ with $1 \leq i \leq 2$ and $3 \leq j \leq 4$. We write $\omega_{ij}$ for their dual. We orient $D$ such that $\omega_{13} \wedge \omega_{14} \wedge \omega_{23} \wedge \omega_{24}$ gives rise to the $G$-invariant volume element on $D$. 

\subsubsection{Special forms for $V$}

The Kudla-Millson form $\varphi_2$ is an element in 
\[
 [\calS(V_{\R}) \otimes \calA^2(D)]^G\simeq
[\calS(V_{\R}) \otimes \wwedge{2} \mathfrak{p}^{\ast}]^K,
\]
where the isomorphism is given by evaluation at the base point. Here $\calA^2(D)$ denotes the differential $2$-forms on $D$. Note that $G$ acts diagonally in the natural fashion. At the base point $\varphi_2$ is given by
\[
\varphi_2= \frac12 \prod_{\mu=3}^4 \sum_{\alpha=1}^{2}  \left( x_{\alpha} - \frac1{2\pi}\frac{\partial}{\partial x_{\alpha}} \right) \varphi_0 \otimes \omega_{\alpha\mu}.
\]
Here $\varphi_0(x) := e^{-\pi(x,x)_{0}}$, where $(x,x)_0= \sum_{i=1}^4 x_i^2$ is the minimal majorant associated to the base point in $D$. Note that $\varphi_2$ has weight $2$, see \cite{KM1}. There is another Schwartz form $\psi_1$ of weight $0$ which lies in $
[\calS(V_{\R}) \otimes \calA^1(D)]^G\simeq
[\calS(V_{\R}) \otimes \mathfrak{p}^{\ast}]^K$ and is given by
\begin{equation}\label{psi20}
\psi_1 =  -x_1x_3\varphi_0(x) \otimes \omega_{14}+x_1x_4  \varphi_0(x) \otimes \omega_{13} - x_2x_3 \varphi_0(x) \otimes \omega_{24}+x_2x_4\varphi_0(x) \otimes \omega_{23}. 
\end{equation}

The key relationship is (see \cite{KM90}, \S 8)

\begin{theorem}\label{localholomorphic1}
\[
\omega(L) \varphi_2 = d \psi_1.
\]
Here $\omega(L)$ is the Weil representation action
of the $\SL_2$-lowering operator $L = \tfrac12 \left(
\begin{smallmatrix}1 & -i \\ -i & -1\end{smallmatrix} \right)  \in \mathfrak{sl}_2(\R)$ on
$\calS(V_{\R})$, while $d$ denotes the exterior differentiation on $D$. 
\end{theorem}

On the upper half plane $\h$, the action of $L$ corresponds to the action of the classical Maass lowering operator which we also denote by $L$. For a function $f$ on $\h$, we have
\[
Lf  = -2iv^2 \frac{\partial}{\partial \bar{\tau}} f.
\]
When made explicit using \eqref{group-tau} Theorem \ref{localholomorphic1} translates to
\begin{equation}\label{partial-d}
v \frac{\partial}{\partial v }  \varphi_2^0(\sqrt{v}x) =  d
\left(\psi_1^0(\sqrt{v}x)\right).
\end{equation}

\subsubsection{The singular form $\tilde{\psi}_1$}

We define the singular form $\tilde{\psi}_1$ by
\begin{align}\label{GreeneqV}
\tilde{\psi}_1(x)  &= - \left( \int_1^{\infty} \psi_1^0(\sqrt{r}x)  \frac{dr}{r} \right)e^{-\pi(x,x)} = - \frac1{2\pi(x_3^2+x_4^2)} \psi_1(x). 
\end{align}
for $x\ne 0$, and as before $\tilde{\psi}^0_{2,0}(x) = \tilde{\psi}_1(x) e^{\pi (x,x)}$ and 
$\tilde{\psi}_1(x,z)$. We see that $\tilde{\psi}_1$ is defined for $x \notin \Span[e_3,e_4]^{\perp}$. Formulated differently, $\tilde{\psi}_1(x,z)$ for fixed $x$ is defined for $z \notin D_x$. Furthermore, as if $\tilde{\psi}_1$ was a Schwartz function of weight $2$, we define
\begin{align}\label{xiZV}
\tilde{\psi}_1(x,\tau,z) &=  \tilde{\psi}_1^0(\sqrt{v}x,z) e^{\pi i (x,x)\tau}  = - \left( \int_v^{\infty} \psi_1^0(\sqrt{r}x,z) \frac{dr}{r} \right) e^{\pi i (x,x)\tau}.
\end{align}

\begin{proposition}\label{schluesselV}

$\tilde{\psi}_1(x,z)$ is a differential $1$-form with singularities along $D_x$. Outside $D_{x}$, we have
\[
d\tilde{\psi}_1(x,z) = \varphi_2(x,z).
\]
Here $d$ denotes the exterior differentiation on $D$. In particular,
for $(x,x)\leq 0$, we see that $\varphi_2(x)$ is exact. 
Furthermore,
\[
L\tilde{\psi}_1(x,\tau) = \psi_1(x,\tau).
\]
\end{proposition}

\begin{proof}

Using \eqref{GreeneqV} and \eqref{partial-d}, we see
\begin{align*}
d \tilde{\psi}_1^0(x,z) &= - \int_1^{\infty}d \left(\psi_1^0(\sqrt{r}x,z)\right)\frac{dr}{r}   =-\int_1^{\infty} \frac{\partial}{\partial r } \left(
\varphi_2^0(\sqrt{r}x,z)\right) \frac{dr}{r}  = \varphi_2^0(x,z),
 \end{align*}
as claimed. The formula $L\tilde{\psi}_1(x,\tau) = \psi_1(x,\tau)$ follows easily from \eqref{xiZV}.
\end{proof}

\begin{remark}
The construction of the singular form $\tilde{\psi}$ works in much greater generality for $\Orth(p,q)$ whenever we have two Schwartz forms $\psi$ and $\varphi$ (of weight $r-2$ and $r$ resp.) such that 
\[
d \psi = L \varphi.
\]
Then the analogous construction of $\tilde{\psi}$ then immediately yields $d \tilde{\psi} = \varphi$ outside a singular set. The main example for this are the general Kudla-Millson forms $\varphi_{q}$ and $\psi_{q-1}$, see \cite{KM90}. For these forms, this construction is already implicit in \cite{BFDuke}. In particular, the proof of Theorem~7.2 in \cite{BFDuke} shows that $\tilde{\psi}$ gives rise to a differential character for the analogous cycle $C_x$, see also Section~\ref{currents} of this paper. The unitary case will be considered in \cite{F-unitary}. 
\end{remark}

\subsection{Schwartz forms for $W$}

Let $W\subset V$ be the rational quadratic space of signature $(1,1)$ obtained from the Witt decomposition of $V$. We will refer to the nullcone of $W$ as the light-cone. We write $\mathfrak{m} \simeq \R$ for the Lie algebra of $M = \SO_0(W_{\R})$. 
Then $X_{23} = e_2 \wedge e_3$ is its natural generator with dual $\omega_{23}$.
We identify the associated symmetric space $D_W$ to $M$ with the
space of lines in $W_{\R}$ on which the bilinear form $(\,,\,)$ is
negative definite:
\[
D_W = \{{\bf s} \subset W_{\R} ; \;\text{$\dim {\bf s} =1$ and $(\,,\,)|_{\bf s} < 0$}
\}.
\]
We pick as base point of $D_W$ the line ${\bf s}_0$ spanned by $e_3$. We set $
x(s) := m(s) e_3 = \sinh(s) e_2 + \cosh(s) e_3$. This realizes the isomorphism $D_W \simeq \R$. Namely, ${\bf s} = \Span x(s)$. Accordingly, we frequently write $s$ for ${\bf s}$ and vice versa. A vector $x \in W$ of positive length defines a point $D_{W,x}$ in
$D$ via $D_{W,x} = \{ {\bf s} \in D; \; {\bf s} \perp x \}$. So ${\bf s} = D_{W,x}$ if and only if $(x,x({\bf s})) =0$. We also write ${\bf s}(x)=D_{W,x}$.

\subsubsection{Special forms for $W$}\label{W-forms}

We carry over the conventions from section~\ref{conventions}. We first consider the Schwartz form $\varphi_{1,1}$ on $W_{\R}$ constructed in \cite{FMcoeff} (in much greater generality) with values in $\calA^1(D_W) \otimes W_{\C}$. More precisely,
\[
\varphi_{1,1} \in [\calS(W_{\R}) \otimes \calA^1(D_W) \otimes W_{\C}]^M \simeq
[\calS(W_{\R}) \otimes \mathfrak{m}^{\ast} \otimes W_{\C}],
\]
Here $M$ acts diagonally on all three factors. Explicitly at the base point, we have
\begin{equation*}
\varphi_{1,1}(x) = \frac{1}{2^{3/2}} \left(4 x_2^2-\frac1{\pi}\right) e^{-\pi (x_2^2+x_3^2)} \otimes \omega_{23} \otimes e_2.
\end{equation*}
Note that $\varphi_{1,1}$ has weight $2$, see \cite{FMcoeff}, Theorem~6.2. We define $\varphi_{1,1}(x,s)$ and $\varphi_{1,1}^0$ as before. There is another Schwartz function $\psi_{0,1}$ of weight $0$ given by
\begin{multline*}
\psi_{0,1}(x) =   -\frac1{\sqrt{2}} x_2x_3 e^{-\pi(x_2^2+x_3^2)}  \otimes 1 \otimes e_2 + \frac{1}{4\sqrt{2}\pi} e^{-\pi(x_2^2+x_3^2)} \otimes 1\otimes e_3 \\ \in
[\calS(W_{\R}) \otimes \wwedge{0} \mathfrak{m}^{\ast} \otimes W_{\C}],
\end{multline*}
and also $\psi_{0,1}(x,s)$ and $\psi_{0,1}^0$. Note that the notation differs from \cite{FMcoeff}, section~6.5. The function $\psi_{0,1}$ defined here is the term $-\psi_{1,1} - \tfrac12 \Lambda_{1,1}$ given in Theorem~6.11 in \cite{FMcoeff}. 
The key relation between $\varphi_{1,1}$ and $\psi_{0,1}$ (correcting a sign mistake in \cite{FMcoeff}) is given by

\begin{theorem}(\cite{FMcoeff}, Theorem~6.2\label{Millson})\label{localholW}
\[
\omega(L) \varphi_{1,1} = d \psi_{0,1}.
\]
\end{theorem}

When made explicit, we have, again using \eqref{group-tau},
\begin{equation}
v^{3/2} \frac{\partial}{\partial v } \left(v^{-1/2} \varphi_{1,1}^0(\sqrt{v}x,s) \right)=  d
\left(\psi_{0,1}^0(\sqrt{v}x,s)\right).
\end{equation}

\subsubsection{The singular Schwartz function $\tilde{\psi}_{0,1}$}

In the same way as for $V$ we define
\begin{align}\label{Greeneq}
\tilde{\psi}_{0,1}(x)  &= - \left(\int_1^{\infty} \psi_{0,1}^0(\sqrt{r}x) r^{-3/2} dr\right)e^{-\pi(x,x)}
\end{align}
for {\it all} $x \in W$, including $x=0$. Define $\tilde{\psi}_{0,1}^0(x)$, $\tilde{\psi}_{0,1}^0(x,s)$ as before and also
\begin{align}\label{xiZ}
\tilde{\psi}_{0,1}(x,\tau,s) &= v^{-1/2} \tilde{\psi}_{0,1}^0(\sqrt{v}x,s) e^{\pi i (x,x)\tau}  = - \left( \int_v^{\infty} \psi_{0,1}^0(\sqrt{r}x,s) r^{-3/2} dr \right) e^{\pi i (x,x)\tau} \notag.
\end{align}
Note that $\tilde{\psi}_{0,1}(x,s)$ has a singularity at $D_{w,x}$. Define functions $A$ and $B$ by
\[
\tilde{\psi}_{0,1}(x)  = A(x)  \otimes 1 \otimes e_2  + B(x) \otimes 1 \otimes e_3
\]
and note 
\begin{equation}\label{AB-eq}
-X_{23} B(x) = A(x).
\end{equation}
We extend these functions to $D_W$ as before. We see by integrating by parts

\begin{lemma}\label{firstformulaforAandB}
\begin{align*}
A(x) &= 
\frac{1}{2\sqrt{\pi}}  x_2 \frac{x_3}{|x_3|}  \Gamma(\tfrac12,2 \pi x_3^2)  e^{-\pi (x,x)}    \label{AA}\\
B(x)& =
 - \frac{1}{2 \sqrt{2} \pi} e^{- \pi(x_2^2+ x_3^2)} + \frac{1}{2  \sqrt{\pi}}|x_3|    \Gamma(\tfrac12,2 \pi x_3^2)  e^{-\pi (x,x)}.
\end{align*}
Here $\G(\tfrac12,a) = \int_a^{\infty} e^{-u} u^{-1/2} du$ is the incomplete $\G$-funtion at $s=1/2$. 
\end{lemma}

It is now immediate that $B$ is continuous and bounded on $D_W$. Since $A$ is clearly bounded
we find that $A$ and $B$ are locally integrable on $D_W$ and integrable and square-integrable on $W$. The singularities of $A$ and $B$ are given as follows.

\begin{lemma}\label{singularitiesofAandB}
\begin{enumerate}
\item[(i)] $B(x) - (1/2)|x_3| e^{-\pi (x,x)} $ is $C^2$ on the Minkowski plane $W$.
\item[(ii)] $A(x)- (1/2) x_2 \frac{x_3}{|x_3|} e^{-\pi (x,x)}$ is $C^1$ on the Minkowski plane $W$.
\end{enumerate}
\end{lemma}

\begin{proof} 
Use Lemma \ref{firstformulaforAandB},
expand the incomplete gamma function around $x_3=0$,
and observe that $|x|x^n$ is $C^n$ for $n>0$. 
\end{proof}

The key properties of $\tilde{\psi}_{0,1}$ analogous to Lemma~\ref{schluesselV} are given by 

\begin{lemma}\label{schluessel}
Outside $D_{W,x}$,
\[
d\tilde{\psi}_{0,1}(x,s) = \varphi_{1,1}(x,s) \qquad \text{and} \qquad 
L\tilde{\psi}_{0,1}(x,\tau) = \psi_{0,1}(x,\tau).
\]
\end{lemma}

\subsubsection{The singular function $\tilde{\psi}_{0,1}'$}

Inspired by \cite{HZ}, section~2.3, we define a functions $A'(x)$ and $B'(x)$ on $W$ by
\begin{align}\label{AB'-eq}
B'(x) &= \begin{cases} \frac12\min(|x_2-x_3|,|x_2+ x_3|)e^{- \pi (x,x)}  & \text{if}  \, x_2^2-x_3^2 >0,\\ 0 \  & \text{otherwise},
\end{cases} \\
A'(x) &= -X_{23}B'(x)=   -\sgn(x_2x_3)B'(x). \notag
\end{align}
 
 \begin{lemma}\label{singularitiesofA'andB'} 
\begin{enumerate}
\item[(i)] 
$B'(x) + \tfrac12|x_3|e^{- \pi (x,x)}$ is $C^2$ on the complement of the light-cone in $W$ and  $C^2$ on nonzero $M$-orbits.
 \item[(ii)] 
 $A'(x) + \tfrac12 x_2 \frac{x_3}{|x_3|}e^{- \pi (x,x)}$ is $C^1$ on the complement of the light-cone in $W$ and  $C^1$  on nonzero $M$-orbits.
 \end{enumerate}
\end{lemma}

We define $\tilde{\psi}'_{0,1}$ by
\[
\tilde{\psi}_{0,1}'(x) = A'(x) \otimes 1 \otimes e_2 +  B'(x) \otimes 1 \otimes e_3
\]
and $\tilde{\psi}_{0,1}'(x,\tau,s) = v^{-1/2} m(s) \tilde{\psi}_{0,1}'(m^{-1}(s)\sqrt{v}x)) e^{\pi i (x,x)\tau}$. A little calculation shows that $\tilde{\psi}_{0,1}'(x)$ is locally constant on $D_W$ with a singularity at $D_{W,x}$ and holomorphic in $\tau$:

\begin{lemma}\label{xi'closed}
Outside $D_{W,x}$ we have
 \[
d \tilde{\psi}_{0,1}'(x) = 0 \qquad \qquad \text{and} \qquad \qquad 
L \tilde{\psi}_{0,1}'(x,\tau) = 0.
\]
\end{lemma}

\begin{remark}
The functions $\tilde{\psi}_{0,1}(x)$ and $\tilde{\psi}_{0,1}'(x)$ define currents on $D_W$. One can show, similarly to Section~\ref{W-currents}, that for $(x,x)>0$ we have
\begin{align*}
d[\tilde{\psi}_{0,1}(x)] = \delta_{D_{W,x} \otimes x} + [\varphi_{1,1}(x)], \qquad \qquad 
d[\tilde{\psi}'_{0,1}(x)] = -\delta_{D_{W,x} \otimes x},
\end{align*}
where $D_{W,x} \otimes x$ is the $0$-cycle $D_{W,x}$ `with coefficient $x \in W$' defined in \cite{FMcoeff}. 
\end{remark}

\subsubsection{The form $\phi_{0,1}$ on $W$}

We now combine $\tilde{\psi}_{0,1}$ and $\tilde{\psi}_{0,1}'$ to obtain an integrable and also square-integrable $W$-valued function 
\[
\phi_{0,1}  \in [L^2(W_{\R}) \otimes \wwedge{0} \mathfrak{m}^{\ast} \otimes W_{\C}]
\]
by
\begin{equation*}
\phi_{0,1}(x) = \tilde{\psi}_{0,1}(x) + \tilde{\psi}_{0,1}'(x)
\end{equation*}
and then also $\phi_{0,1}(x,s)$. Combining Lemmas \ref{singularitiesofAandB}, \ref{singularitiesofA'andB'} and \eqref{AB-eq}, \eqref{AB'-eq} we obtain

\begin{proposition}\label{phi-prop}
\begin{itemize}
\item[(i)] $B(x) + B'(x)$ is $C^2$ on the complement of the light-cone in $W$ and  $C^2$ on nonzero $M$-orbits.
\item[(ii)] $A(x) + A'(x)$ is $C^1$ on the complement of the light-cone in $W$ and  $C^1$ on nonzero $M$ orbits.
\item[(iii)] $X_{23}(B + B') = -(A + A')$ on all of $W$.
\end{itemize}
So for given $x$, the function $\phi_{0,1}(x,s)$ is a $C^1$-function on $D_W$ with values in $W_{\C}$.
\end{proposition}

The following theorem is fundamental for us. It is an immediate consequence of the Lemmas~\ref{schluessel} and \ref{xi'closed}. 

\begin{theorem}\label{local-phi}
The form $\varphi_{1,1}$ on $D_W$ is exact. Namely,
\[
d \phi_{0,1} = \varphi_{1,1}.
\]
Furthermore,
\[
L \phi_{0,1} = d \psi_{0,1}.
\]
\end{theorem}

\begin{proposition}
The function $\phi_{0,1}$ is an eigenfunction of $K'=\SO(2)$ of weight $2$ under the Weil representation. More precisely,
\[
\omega(k') \phi_{0,1} =  \chi^2(k')\phi_{0,1},
\]
where $\chi$ is the standard character of $\SO(2) \simeq U(1)$.
\end{proposition}

\begin{proof}
It suffices to show this for one component of $\phi_{0,1}$, that is, the function $B(x)+B'(x)$. Then the assertion has been already proved in \S 2.3 by showing that $B(x)+B'(x)$ is an eigenfunction under the Fourier transform. We give here an infinitesimal proof. Since $\omega(k')$ acts essentially as Fourier transform and $B+B'$ is $L^1$, we see that $\omega(k')(B+B')$ is continuous. Hence it suffices to establish the corresponding current equality $[\omega(k')(B+B')] = \chi^2(k')[B+B']$, since continuous functions coincide when they induce the same current. The infinitesimal generator of $K'$ acts by $
H:=\frac{-i}{4\pi} \left(\tfrac{\partial^2}{\partial{x_2^2}} - \tfrac{\partial^2}{\partial{x_3^2}}\right) + \pi i (x_2^2-x_3^2)$,
and a straightforward calculation immediately shows
\[
H B' = 2i B' \qquad \text{and} \qquad  H {B} = 2i {B},
\]
outside the singularity $x_2^2-x_3^2=0$. 
\begin{comment}
For $B$, we have
\[
B(x) =- \frac{1}{2 \sqrt{2} \pi} e^{-\pi(x_2^2+ x_3^2)} + \frac{1}{2  \sqrt{\pi}}|x_3| \Gamma(\tfrac12,2 \pi x_3^2) e^{ -\pi(x_2^2- x_3^2)},
\]
and we denote the second summand by $\tilde{B}(x)$. One easily sees that the first summand is annihilated by 
$\frac{-i}{4\pi} \square + \pi i r^2$. For $\tilde{B}(x)$, we first see 
\[
\frac{\partial}{\partial x_3} \tilde{B}(x) = \frac{1}{x_3} \tilde{B}(x) -\sqrt{2}x_3  e^{-\pi(x_2^2+ x_3^2)}.
+ 2\pi x_3 \tilde{B}(x),
\]
again away from the singularity $x_3=0$. Indeed, this follows easily from $\frac{\partial}{\partial x_3}  \Gamma(\tfrac12,2 \pi x_3^2) = - 2 \sqrt{2\pi} \sgn(x_3) e^{-2 \pi x_3^2}$. A little calculation then gives
\begin{align*}
\frac{\partial^2}{\partial x_3^2} \tilde{B}(x) = 6 \pi \tilde{B}(x) - 2\sqrt{2} e^{ -\pi(x_2^2+x_3^2)} +(2\pi)^2 x_3^2 \tilde{B}(x).
\end{align*}
Then 
\[
\frac{-i}{4\pi} \square \tilde{B}(x) = -\pi i r^2 \tilde{B}(x) + 2i \tilde{B}(x) - 2i \frac{1}{2\sqrt{2}\pi} e^{ -\pi(x_2^2+x_3^2)}.
\]
In conclusion, 
\[
H {B}(x) = 2i {B}(x),
\]
again, outside the singularity. 
\end{comment}
Now we consider the currents $H[B]$ and $H[B']$. An easy calculation using that $B$ and $B'$ are $C^2$ up to $|x_3|e^{-\pi(x_2^2-x_3^2)}$  shows that for a test function $f$ on $W$ we have
\begin{align*}
H[B](f) &= [HB](f) + \int_{0}^{\infty} e^{-\pi x_2^2} f(x_2,0)dx_2, \\
H[B'](f) &= [HB'](f) - \int_{0}^{\infty} e^{-\pi x_2^2} f(x_2,0) dx_2.
\end{align*}
Thus $H[B+B'] = [H(B+B')]= 2i[B+B']$ as claimed.
\end{proof}

\subsubsection{The map $\iota_P$}\label{iotaP}

We define a map 
\[
\iota_P: \calS(W_{\R}) \otimes \wwedge{i} \mathfrak{m}^{\ast} \otimes W_{\C} \to \calS(W_{\R}) \otimes \wwedge{i+1} \left( \mathfrak{m}^{\ast} \oplus  \mathfrak{n}^{\ast} \right) 
\]
by 
\[
\iota_P(\varphi \otimes \omega \otimes w) = \varphi \otimes \left((\omega \wedge (w \wedge u')\right).
\]
Here we used the isomorphism $\mathfrak{n} \simeq W \wedge \R u \in \bigwedge^{2} V_\R \simeq \mathfrak{g}$ and identify $W$ with its dual via the bilinear form $(\,,\,)$ so that $\mathfrak{n}^{\ast} \simeq W \wedge \R u'$. In \cite{FMres}, Section ~6.2 we explain that $\iota_P$ is a map of Lie algebra complexes. Hence we obtain a map of complexes
\[
[\calS(W_{\R}) \otimes \calA^i(D_W) \otimes W_{\C}]^M \to [\calS(W_{\R}) \otimes \calA^{i+1}(e(P))]^{NM}, 
\]
which we also denote by $\iota_P$. Here $N$ acts trivially on $\calS(W_{\R})$. Explicitly, the vectors $e_2$ and $e_3$ in $W$ map under $\iota_P$ to the left-invariant $1$-forms
\[
e_2 \mapsto \cosh(s)dw_2 - \sinh(s)dw_3 \qquad e_3 \mapsto \sinh(s)dw_2 - \cosh(s)dw_3
\]
with the coordinate functions $w_2,w_3$ on $W$ defined by $w=w_2e_2+w_3e_3$. We apply $\iota_P$ to the forms on $W$ of this section, and we obtain $\varphi_{1,1}^P$, $\phi_{0,1}^P$, $\psi_{0,1}^P$, and ${\psi'}_{0,1}^P$.

\section{The boundary theta lift and linking numbers in Sol}

\subsection{Global theta functions for $W$}

We let $\calL_W$ be a $\G_P$-invariant (coset of a) lattice in $W$, where $\G_N$ acts trivially on $W$. For $\varphi_{1,1}$, we define its theta function by 
\[
\theta_{\varphi_{1,1}}(\tau,{\calL_W})= \sum_{x \in \calL_W} \varphi_{1,1}(x,\tau)
\]
and similarly for $\psi_{0,1}$, and $\phi_{0,1}$. Then the usual theta machinery gives that
$\theta_{\varphi_{1,1}}(\tau,{\calL_W})$ and $\theta_{\phi_{0,1}}(\calL_W)$ both transform like (non)-holomorphic modular forms of weight $2$ for some congruence subgroup of $\Sl_2(\Z)$. 

\begin{remark}
The claim is not obvious for $\theta_{\phi_{0,1}}$, since $\phi_{0,1}$ is not a Schwartz function. In that case, we use Proposition~\ref{phi-prop}. The component $B+B'$ of $\phi_{0,1}$ is $C^2$ outside the light cone. Since $W$ is anisotropic we can then apply Possion summation, and this component transforms like a modular form. Then apply the differential operator $X_{23}$ to obtain the same for the other component $A+A'$ of $\phi_{0,1}$. 
 
 In fact, if $W$ is isotropic and $\calL_W$ intersects non-trivially with the light cone, then $\theta_{\phi_{0,1}}$ is not quite a modular form. The case, when the $\Q$-rank of $V$ is $2$ is interesting in its own right. We will discuss this elsewhere. 
\end{remark}

Via the map $\iota_P$ from Section~\ref{iotaP} we can view all theta functions for $W$ as functions resp. differential forms on $e'(P)$. We set $
\theta^P_{\varphi_{1,1}} =   \theta_{\varphi_{1,1}^P}$, 
and similarly $\theta^P_{\psi_{0,1}}$ and $\theta^P_{\phi_{0,1}}$. Since $\iota_P$ is a map of complexes we immediately see by Theorem~\ref{localholW} and Theorem~\ref{local-phi}

\begin{proposition}\label{globalholomorphic2}
\[
L \theta^P_{\varphi_{1,1}} = d\theta^P_{\psi_{0,1}} \qquad \text{and} \qquad 
L \theta^P_{\phi_{0,1}} = \theta^P_{\psi_{0,1}}.
\]
\end{proposition}

We now interpret the (holomorphic) Fourier coefficients of the boundary theta lift associated to $ \theta^P_{\phi}(\tau,\calL_{W_P})$. They are given by linking numbers. We have

\begin{theorem}\label{xi'-integralP}
Let $c$ a homological trivial $1$-cycle in $e'(P)$ which
is disjoint from the torus fibers containing components of $\partial C_n$ or for $c=\partial C_y$ for $C_y$ one of the components of $C_n$, we have
\[
\int_{c} \theta^P_{\phi_{0,1}}(\tau,\calL_{W_P}) = 
  \sum_{n=1}^{\infty}  \Lk((\partial C_n)_P,c) q^n
 \; + \; \sum_{n \in \Q} \int_{c}  {\tilde{\psi}_{0,1}^P}(n)(\tau).
\]
 So the Fourier coefficients of the holomorphic part of $\int_{c} \theta^P_{\phi}(\tau,\calL_{W_P})$ are the linking numbers of the cycles $c$ and $ \partial C_n$ at the boundary component $e'(P)$.  
 \end{theorem}

Theorem \ref{xi'-integralP} follows from $\phi_{0,1} = \tilde{\psi}_{0,1} + \tilde{\psi'}_{0,1}$ combined with Theorem~\ref{linking-dual} below.  

\begin{example}\label{HZbeta}
In the situation of Examples~\ref{HZex} and \ref{LinkCnCmex}, we obtain
\[
\int_{\partial C_1} \theta^P_{\phi_{0,1}}(\tau,\calL_{W_P}) = \frac{1}{\sqrt{2d}} \sum_{\substack{\la \in \mathcal{O}_K \\ \la\la'>0}} \min(|\la|,|\la'|) e^{-2\pi \la \la' \tau } - \frac{\sqrt{2}}{\sqrt{dv}} \sum_{\la \in \mathcal{O}_K} \beta (\pi v (\la-\la')^2) e^{-2\pi \la \la' \tau},
\]
where $\beta(s) = \tfrac1{16\pi} \int_1^{\infty} e^{-st}t^{-3/2} dt$. This is (up to a constant) exactly Zagier's function $\mathcal{W}(\tau)$ in \cite{HZ}, \S 2.3.

\end{example}

\subsection{Linking numbers, de Rham cohomology and linking duals} \label{generallinking}
We begin with a general discussion of integral formulas for linking numbers.
Such formulas go back to the classical Gauss-Amp\`ere formula for $\R^3$, see \cite{F}, p.79-81, and \cite{DG} for its generalization to $S^3$ and $H^3$. 
Suppose now that $c$ is a $1$-cycle in an oriented compact $3$-manifold M that is a rational boundary and $U$ is a tubular neighborhood of $c$. 
\begin{definition}
We will say any closed form $\beta$ in $M-U$ is a {\it linking dual} (relative to $U$) of the bounding $1$-cycle $c$ if for any $1$-cycle $a$ in $M-U$
which is a rational boundary in $M$ we have
\begin{equation*} \label{linkingdualdef} 
\int_a \beta = \Lk(a,c).
\end{equation*}
\end{definition}

We will prove that given a cycle $c$ that bounds rationally then linking duals for $c$  exist for all tubular neighborhoods $U$ of $c$.  Let $\eta$ be a Thom form for $c$
compactly supported in $U$. This means that $\eta$ is closed and has integral $1$ over any normal disk to $c$. Let $\eta_M$ be the extension of $\eta$ to $M$ by zero.  It is standard in topology (the extension of the Thom class by zero is the Poincar\'e dual of the zero section of the normal bundle) that the form  $\eta_M$ represents the $2$-dimensional cohomology class on $M$ which is Poincar\'e dual to $c$. Since $c$ is a rational boundary there exists a $1$-form $\beta$ on $M$ such that $d \beta = \eta_M$. We will now see that $\beta$ is a linking dual of $c$.  To this end,  suppose  $a$ is a $1$-cycle in $M -U$ which is a rational boundary in $M$, hence there exists a rational
chain $A$ with $\partial A = a$.  We may suppose $\eta_M$ vanishes in a neighborhood $V$ of $a$ which is disjoint from $U$ .  Then the restriction $\eta_{M-V}$ of $\eta_M$ to $M-V$ represents the (relative) Poincar\'e dual of the absolute cycle $c$ in $(M-V, \partial (M-V))$. Using this
the reader will show that
\begin{equation*} \label{firstequation}
 \int_A \eta_M = \int_{A \cap (M-V)} \eta_{M-V}  = A \cdot c = \Lk(a,c).
\end{equation*}
 Note that restriction of $\beta$ to $M-U$ is closed. Then
\[
\int_a \beta  = \int_A \eta_M = \Lk(a,c).
\]
Hence we have 
\begin{proposition} \label{linkingdualprop}
$\beta$ is a linking dual of $c$. 
\end{proposition}

\subsection{The $1$-form $e^{2 \pi n} \tilde{\psi'}_{0,1}(n)$ is a linking dual of $(\partial C_n)_P$}
We now return to the case in hand. In what follows, we drop subscript and superscript $P$'s since we are fixing a boundary component $e(P)$.
We let $F_n$ be the union of the fibers containing components of $\partial C_n$, and we let $F_x$ be the fiber containing $c_x$. Recall that $c_x$ is the image of $D_x \cap e(P)$ in $e'(P)$. 

\begin{theorem}\label{linking-dual}
Let $n>0$. The $1$-form $e^{2 \pi n} \tilde{\psi'}_{0,1}(n)$ is a linking dual for $\partial C_n$ in $e'(P)$ relative to any neighborhood
$U$ of $F_n$.  Hence, for $c$ a rational $1$-boundary in $e'(P)$ which is disjoint from $F_n$ we have
\begin{equation}\label{Linkingdualintegral}
\int_c  \tilde{\psi'}_{0,1}(n) =  \Lk(\partial C_n,c) e^{- 2 \pi n}. 
\end{equation}
Furthermore, \eqref{Linkingdualintegral} holds when $c=c_y$ contained in one fiber $F_x$ of $\partial C_n$.
\end{theorem}

We will first deal with the case in which $c$ is disjoint from $F_n$ (which we will refer to in what follows as  case (i)), then at the end of this section  we will reduce the case in which $c=c_y$ (which we will refer to as case (ii)) to case (i) by a Stokes' Theorem argument.
Thus we will now assume we are in case (i).

The key step is 

\begin{proposition}\label{finalintegral}
Let $n>0$ and let $\eta$ be an exact $2$-form in $e'(P)$ which is compactly supported in the complement of $F_n$. Then 
\begin{equation}\label{integraletapsi}
 \int_{e'(P)} \eta \wedge \tilde{\psi'}_{0,1}(n) = \left(\int_{A_n} \eta\right) e^{ - 2 \pi n}. 
\end{equation}
\end{proposition}
\begin{remark}\label{youwillneedthis}
Note that \eqref{integraletapsi} also holds in case $\eta =\Omega_P$. In this case the right-hand side is zero by the normalization of the cap $A_n$ and the left-hand side is zero because $\Omega \wedge \tilde{\psi'}_{0,1}(n) =0$ since $\Omega$ has bidegree $(0,2)$
and $\tilde{\psi'}_{0,1}(n)$ has bidegree $(0,1)$ (here we use the obvious base/fiber bigrading on the de Rham algebra of $e'(P)$). 
\end{remark} 

\subsection{Proof of Proposition~\ref{finalintegral}}\label{8.2}

\begin{lemma}\label{LemmaA}
Under the hypothesis on $\eta$ in Proposition~\ref{finalintegral} we have
\[
 \int_{A_n} \eta = \sum_{ \substack{x\in \G_M \back \mathcal{L}_W \\ (x,x)=2n}} \min_{\la \in \Lambda_W} {\hspace{-5pt}'} |(\la,x)| \int_{a_x} \eta
\]
\end{lemma}

\begin{proof}
We use Lemma~\ref{LemmaB}. Write $\eta = d \omega$ for some $1$-form $\omega$ which by the support condition on $\eta$ is closed in $F_n$. Since $c_{x+ku}$ and  $c_x$ are parallel hence homologous circles in $F_x$, we see $\int_{a_{x+ku}} \eta = \int_{c_{x+ku}} \omega = \int_{c_x} \omega = \int_{a_x} \eta$. 
\end{proof}

Since 
\[
\tilde{\psi'}_{0,1}(n) = \sum_{ \substack{x\in \G_M \back \mathcal{L}_W \\ (x,x)=2n}} \sum_{ \g \in \G_M} \g^{\ast} \tilde{\psi'}_{0,1}(x),
\]
Proposition~\ref{finalintegral} will now follow from

\begin{proposition}
Under the hypothesis on $\eta$ in Proposition~\ref{finalintegral}, we have for any positive length vector $x \in \calL_W$
\[
\int_{e'(P)} \eta \wedge  \sum_{ \g \in \G_M} \g^{\ast} \tilde{\psi'}_{0,1}(x) = (\min_{\la \in \Lambda_W}   {\hspace{-2pt}'}|(\la,x)|) \left(\int_{a_x} \eta \right) e^{-\pi (x,x)}.
\]
\end{proposition}

\begin{proof}
By choosing appropriate coordinates we can assume that $x = \mu e_2$ with $\mu = \pm \sqrt{2n}$, so that the singularity of $\sum_{ \g \in \G_M} \g^{\ast} \tilde{\psi'}_{0,1}(x)$ in $e'(P)$ occurs at $s=0$. 
We pick a tubular neighborhood $U_\eps= (-\eps,\eps) \times T^2$ in $e'(P)$ around $F_x$. Then we have first
\[
\int_{e'(P)} \eta \wedge  \sum_{ \g \in \G_M} \g^{\ast} \tilde{\psi'}_{0,1}^P(x)= 
  \lim_{\epsilon \to 0} \int_{e'(P) - U_{\eps}} \eta \wedge   \sum_{ \g \in \G_M} \g^{\ast}\tilde{\psi'}_{0,1}^P(x).
 \]
Since $\eta \wedge \tilde{\psi'}_{0,1}(x) = d(\omega \wedge \tilde{\psi'}_{0,1}(x))$ outside $U_{\eps}$ and $\partial (e'(P) - U_{\eps}) =
- \partial U_{\eps}$   we see by Stokes' theorem
\begin{align}\label{calcAA1}
&\int_{e'(P) - U_{\eps}} \eta \wedge  \sum_{ \g \in \G_M} \g^{\ast} \tilde{\psi'}_{0,1}(x)
= - \int_{\partial U_{\eps}} \omega \wedge   \sum_{ \g \in \G_M} \g^{\ast}\tilde{\psi'}_{0,1}(x) \\
& \quad = \sum_{ \g \in \G_M}  \int_{T^2} \left[\omega(-\eps,w) \wedge \tilde{{\psi'}_{0,1}}(\g^{-1}x,-\epsilon,w) - 
\omega(\eps,w) \wedge \tilde{\psi'}_{0,1}(\g^{-1} x,\epsilon,w) \right]. \notag
\end{align}
For $\g \ne 1$ we note that $\omega(s,w) \wedge \tilde{\psi'}_{0,1}(\g^{-1}x,s,w)$ is continuous at $s=0$, while for $\g=1$, we have
\begin{equation}\label{psi'formula}
 \tilde{\psi'}_{0,1}(\mu e_2,s,w) =  \frac12|\mu| (\sgn(s)dw_2-dw_3) e^{- \pi \mu^2}.
 \end{equation}
Hence taking the limit in the last term of  \eqref{calcAA1} we obtain 
\[
|\mu| e^{- \pi \mu^2}\int_{T^2} \omega_3(0,w) dw_2dw_3 = |\mu|  e^{- \pi \mu^2}\int_{T^2/ c_{e_2}} \left( \int_{c_{e_2}} \omega(0,w_2,w_3) \right) dw_2.
 \]In the expression $T^2/ c_{e_2}$ (and for the rest of this proof)  we have abused notation and identified the cycle $c_{e_2}$ with the subgroup $0 \times S^1$ of $T^2$.

Here $\omega_3$ is the $dw_3$ component of $\omega$ and we used that $\partial D_{x}$ is the $w_3$-line in $W$. Note that the inner integral on the right is the period 
of $\omega$ over (homologous) horizontal translates of the cycle $c_{e_2}$. But the restriction of $\omega$ to $F_x$ is closed so $\int_{c_{e_2}} \omega(0,w_2,w_3)$ is independent of $w_2$ and the last integral becomes
$\left(  \int_{T^2/ c_{e_2}} dw_2 \right)\left( \int_{c_{e_2}} \omega \right)e^{- \pi \mu^2}$. 
 But $\int_{c_{e_2}} \omega = \int_{A_{e_2}} \eta$. The proposition is then a consequence of
\[
|\mu| \int_{T^2/ \partial c_{e_2}} dw_2  = |\mu|  \min_{\la \in \Lambda_W}  {\hspace{-2pt}'} |(\la,e_2)| = \min_{\la \in \Lambda_W}  {\hspace{-2pt}'} |(\la,x)|,
\] 
which follows from the fact that the map $W \to \R$ given by $w \mapsto (w,e_2)$ induces an isomorphism $T^2/ \partial C_{e_2} \simeq \R / \min_{\la \in \Lambda_W}'|(\la,e_2)|)\Z$.  

\end{proof}

\subsection{Proof of Theorem~\ref{linking-dual}}\label{W-currents}

We now prove Theorem~\ref{linking-dual}. First we will assume that we are in case(i). We need to show  
\[
\int_c  \tilde{\psi'}_{0,1}(n) = \Lk(\partial C_n,c) e^{-2\pi n} =  (A_n \cdot c) e^{-2\pi n}.
\]
The theorem will be a consequence of the following discussion.  
We may assume that $c$ is an embedded loop in $e'(P)$ (note that since any loop in a manifold of dimension $3$ or more is homotopic to an embedded loop by transversality any homology class of degree $1$ in $e'(P)$ is represented by an embedded loop). 

Choose a tubular neighborhood $N(c)$ of $c$ such that $N(c)$ is disjoint from $F_n$. Let $\eta_c$ be a closed $2$-form which is supported inside $N(c)$ and has integral $1$ on the disk fibers of $N(c)$ (a Thom class for the normal disk bundle $N(c)$).
Then we have proved in Subsection \ref{generallinking}  

\begin{lemma}\label{integrallink}
\begin{equation} \label{secondformula}
\int_{A_n} \eta_c= A_n \cdot  c= \Lk(\partial C_n, c).
\end{equation}
\end{lemma}

We then have

\begin{lemma} \label{laststep}
\begin{equation*} \label{fifthformula}
\int_c \tilde{\psi'}_{0,1}(n) = \left(\int_{A_n} \eta_c\right) e^{- 2\pi n}.
\end{equation*}
\end{lemma}

\begin{proof}
To prove the Lemma we compute $\int_{e'(P)} \eta_c \wedge  \tilde{\psi'}_{0,1}(n)= \int_{e'(P)} \tilde{\psi'}_{0,1}(n) \wedge \eta_c$ in two different ways. First we apply Proposition \ref{finalintegral} with $\eta = \eta_c$.  We deduce
\begin{equation*}\label{thirdformula}
\int_{e'(P)}\eta_c \wedge \tilde{\psi'}_{0,1}(n) = \left(\int_{A_n} \eta_c\right) e^{-2 \pi n}.
\end{equation*}
Next choose a tubular neighborhood $V_n$ of the fibers $F_n$   such that $e'(P) - V_n$
contains $N(c)$.  Then $\tilde{\psi'}_{0,1}(n)$ is smooth on $e'(P) - V_n \supset \supp (\eta_c)$.  Also, since $\eta_c$ is the extension of a Thom class by zero,  the restriction of  $\eta_c$ to $e'(P) - V_n$ represents the Poincar\'e dual $PD(c)$ of the absolute cycle $c$ in 
$e'(P) -V_n$. The lemma now follows from 
\begin{equation*}
\int_{e'(P)} \tilde{\psi'}_{0,1}(n) \wedge \eta_c   = \int_{e'(P)- V_n} \tilde{\psi'}_{0,1}(n) \wedge \eta_c = \int_{e'(P)- V_n} \tilde{\psi'}_{0,1}(n) \wedge PD(c)
  = \int_{c} \tilde{\psi'}_{0,1}(n). 
\end{equation*}
\end{proof}

By Lemma \ref{integrallink} this concludes the proof of Theorem~\ref{linking-dual} in the case when $c$ is disjoint from the fibers $F_n$.

It remains to treat case (ii). 
Thus we now assume that $c=c_y$ which is contained in a fiber $F_x$ containing a component  of $\partial C_n$.  We first prove
\begin{lemma}\label{selflinkingforx}
\[
\int_{c} \widetilde{\psi}'_{0,1}(x) =
\int_{c(\epsilon)}  \widetilde{\psi}'_{0,1}(x).
\]
\end{lemma}
\begin{proof} 
We can take $x = \mu e_2$ and hence $c$ is contained in the fiber over the image of $e_3 \in W$. Hence, by Proposition \ref{boundaryofC}, $c $ is the circle in the torus fiber at $s(x )=0$ {\it in the $e_3$-direction}, i.e., parallel to the image of $(0,\R e_3)$ in $e'(P)$. We note that by \eqref{psi'formula} even though $\widetilde{\psi}'_{0,1}(x)$ is not defined on the whole fiber over $s=0$ its restriction to $c$ is smooth. Hence the left hand side is well-defined since all the other terms in the sum are defined on the whole fiber and in fact in a neighborhood of that fiber. Hence the locally constant form $\sum_{\gamma \in \Gamma_M} \gamma^* \widetilde{\psi}'_{0,1}(x)$ is closed on the cylinder $[0, \epsilon] \times c $, and its integrals over the circles $s \times c $ all coincide. But $ \eps \times c = c(\epsilon)$. The lemma follows.
\end{proof}
Summing over $x$ and using case (i) we obtain 
\[
\int_{c} \widetilde{\psi}'_{0,1}(n) =\int_{c(\epsilon)}  \widetilde{\psi}'_{0,1}(n) = \Lk(\partial C_n ,c(\epsilon)),
\]
since $c(\epsilon)$ is disjoint from all the components of $F_n$. Thus it suffices to prove 
\begin{equation}\label{lastlinkingstep}
\Lk(\partial C_n, c) = \Lk(\partial C_n, c(\epsilon)).
\end{equation}
To this end suppose that $c \subset F_x \subset F_n$ and $c_1,\cdots,c_k$ are the components of $\partial C_n$ contained in $F_x$. Hence
$c$ and $c_i,1 \leq i \leq k$, are all parallel. Since the fibers containing all other components of $\partial C_n$ are disjoint from $c$, \eqref{lastlinkingstep} will follow from
\[
\Lk(c_i,c) = \Lk(c_i, c(\epsilon)), 1 \leq i \leq k.
\]
If $c_i = c$ then the previous equation is the definition of $Lk(c,c)$. Thus we may assume $c_i$ is parallel to and disjoint from $c$. 
In this case their linking number is already topologically defined.
But since $c$ is disjoint from $c_i$ the circles $c$ and $c(\eps)$ are homologous in the complement of $c_i$ (by the product homology $c \times [0,\eps]$) and since the linking number with $c_i$ is a homological invariant of the complement of $c_i$ in $e'(P)$ we have 
$\Lk(c_i, c(\eps)) =\ Lk(c_i, c)$.

We this Theorem~\ref{linking-dual} is proved.

\section{The generating series of the capped cycles}

In this section, we show that the generating series of the `capped' cycles $C_n^c$ gives rise to a modular form, extending Theorem~\ref{KM90} to a lift of the full cohomology $H^2(X)$ of $X$. In particular, we give our new proof of the theorem of Hirzebruch and Zagier and show how a remarkable feature of their proof appears from our point of view.

\subsection{The theta series associated to $\varphi_2$}

We define the theta series
\[
\theta_{\varphi_2}(\tau,\calL) = \sum_{x \in \calL} \varphi_2(x,\tau,z).
\]
 In the following we will often drop the argument $\calL = L+h$. For $n \in \Q$, we also set 
\[
\varphi_2(n) = \sum_{n \in \calL_n, x \ne0} \varphi_2(x).
\]
Clearly, $\theta_{\varphi_2}(\tau,\calL)$ and $\varphi_2(n)$ descend to closed differential $2$-forms on $X$. Furthermore, $\theta_{\varphi_2}(\tau,\calL)$ is a non-holomorphic modular form in $\tau$ of weight $2$ for the principal congruence subgroup $\G(N)$. In fact, for $\mathcal{L} = L$ as in Example~\ref{HZex}, $\theta_{\varphi_2}(\tau,\calL)$ transforms like a form for $\G_0(d)$ of nebentypus. 

\begin{theorem}[Kudla-Millson \cite{KM90}]\label{KM90}
We have
\[
[\theta_{\varphi_2}(\tau)] =  -\frac{1}{2\pi}\delta_{h0} [\omega] + \sum_{n>0} \PD[C_n] q^n \in H^2(X,\Q) \otimes M_2(\G(N)).
\]
That is, for any closed $2$-form $\eta$ on $X$ with compact support,
\[
\Lambda(\eta,\tau) := \int_X \eta \wedge \theta_{\varphi_2}(\tau,\calL)= -\frac{1}{2\pi}\delta_{h0} \int_X \eta \wedge \omega + \sum_{n>0} \left( \int_{C_n} \eta \right)q^n.
\]
Here $\delta_{h0}$ is Kronecker delta, and $\omega$ is the K{\"a}hler form on $D$ normalized such that its restriction to the base point is given by $\omega_{13}\wedge \omega_{14}+\omega_{23}\wedge \omega_{24}$. 
We obtain a map
\begin{equation}
\Lambda: H_c^{2}(X,\C) \to M_{2}(\G(N))
\end{equation}
from the cohomology with compact supports to the space of holomorphic modular forms of weight $2$ for the principal congruence subgroup $\G(N) \subset \SL_2(\Z))$. Alternatively, for $C$ an absolute $2$-cycle in $X$ defining a class in $H_2(X,\Z)$, the lift $\Lambda(C,\tau)$ is given by \eqref{KM-id} with $C_0$ the class given by $-\frac{1}{2\pi}\delta_{h0} [\omega]$.
\end{theorem}

The key fact for the proof of the Fourier expansion is that for $n>0$, the form $\varphi_2(n)$ is a Poincar\'e dual form of $C_n$, while $\varphi_2(n)$ is exact for $n \leq 0$, see also Section~\ref{currents}.

\subsection{The restrictions of the global theta functions}

\begin{theorem}\label{restriction}

The differential forms $\theta_{\varphi_2}(\calL_V)$ and $\theta_{\psi_1}(\calL_V)$ on $X$ extend to the Borel-Serre compactification $\overline{X}$. More precisely, for the restriction $i_P^{\ast}$ to the boundary face $e'(P)$ of $\overline{X}$, we have
\[
i_P^{\ast} \theta_{\varphi_2}(\calL_V) = \theta^P_{\varphi_{1,1}}(\calL_{W_P}) \qquad \text{and} \qquad i_P^{\ast} \theta_{\psi_1}(\calL_V) =  \theta^P_{\psi_{0,1}}(\calL_{W_P}). 
\]
\end{theorem}

\begin{proof}
The restriction of $\theta_{\varphi_2}(\calL_V)$ is the theme (in much greater generality) of \cite{FMres}. For $\theta_{\psi_1}(\calL_V)$ one proceeds in the same way. In short, one detects the boundary behaviour of the theta functions by switching to a mixed model of the Weil representation. For a model calculation see the proof of Theorem~\ref{psitilderes} below. 
\end{proof}

We conclude by Proposition~\ref{globalholomorphic2}

\begin{theorem}\label{globalexact}

The restriction of $\theta_{\varphi_2}(\calL_V)$ to the boundary of $\overline{X}$ is exact and
\[
i_P^{\ast} \theta_{\varphi_2}( \calL_V) = d\left(  \theta^P_{\phi_{0,1}}(\calL_{W_P}) \right). 
\]
\end{theorem}

We also have a crucial restriction result for the singular form ${ \tilde{\psi}_{0,1}}$. However, one needs to be careful in forming the naive theta series associated to $\tilde{\psi}_{0,1}$ by summing over all (non-zero) lattice elements. This would give a form on $X$ with singularities on a dense subset of $X$. Instead we define $\tilde{\psi}_{{2,0}}(n)$ in the same way as for $\varphi_2(n)$ by summing over all non-zero $x \in \calL_V$ of length $n \in \Q$. This gives a $1$-form on $X$ which for $n>0$ has singularities along the locally finite cycle $C_n$. Similarly, we define
\[
\tilde{\psi}_{0,1}^P(n)= \sum_{\substack{x \in \calL_{W_P}, (x,x)=2n}} \tilde{\psi}_{0,1}^P(x),
\]
\vskip-.2cm
\noindent
which descends to a $1$-form on $e'(P)$ with singularities. We also define $\tilde{\psi'}_{0,1}(n)$ 
and $\phi_{0,1}^P(n)$ in the same way. We have

\begin{proposition}\label{psitilderes}
The restriction of the $1$-form $\tilde{\psi}_{{1}}(n)$ to $e'(P)$ is given by 
\[
i_P^{\ast} \tilde{\psi}_1(n) = \tilde{\psi}_{0,1}^P(n).
\]
\end{proposition}

\begin{proof}
We assume that $P$ is the stabilizer of the isotropic line $\ell=\Q u$. For $x = au + x_W + bu'$, we have for the majorant at $z=(w,t,s)$ the formula
\[
(x,x)_z = \frac1{t^2}(a-(x_W,w)-bq(w))^2 + (x_w+bw,x_w+bw)_s +b^2t^2.
\]
Here $(\,,\,)_s$ is the majorant associated to $W$. Hence by \eqref{GreeneqV} and \eqref{psi20} we see that the sum of all $x \in \calL_V$ with $b \ne 0$ in $\tilde{\psi}_1(n)$ is uniformly rapidly decreasing as $t \to \infty$. Now fix an element $x_W \in \calL_W$. Then  $x_W +(a+h)u \in \calL_V$ for all $a \in \Z$ for some $h \in \Q/\Z$; in fact all elements in $\calL_V \cap u^{\perp}$ are of this form. We consider $\sum_{a \in \Z} \tilde{\psi}_1(x_W +(a+h)u,z)$ as $t \to \infty$. By considerations as in \cite{FMres}, sections 4 and 9, we can assume $w=0$ and $s=0$. We apply Poisson summation for the sum on $a \in \Z$ and obtain
\begin{align*}
\sum_{a \in \Z} \tilde{\psi}_1(x_W +au,z) = \sum_{k \in \Z} \left( \int_{1}^{\infty} P(x,t,r) e^{-2\pi x_3^2r +t^2k^2/r} \frac{dr}{r} \right) e^{-2\pi i k h} e^{-\pi (x_W,x_W)},
\end{align*}
where
\[
P(x,t,r) = \frac{x_2x_3\sqrt{r}}{\sqrt{2}} dw_2 + \frac{1}{2\sqrt{2}}\left(\frac1{2\pi}-\frac{t^2k^2}{r}\right)dw_3 - \frac{i x_3 k}{\sqrt{2}}dt +  \frac{i x_2kt}{\sqrt{2}} ds.\]
Now the sum over all $k \ne 0$ is rapidly decreasing while for $k=0$ we obtain $\tilde{\psi}_{0,1}(x_W)$. If $x_W=0$, i.e., for $n=0$ one needs to argue slightly differently. Then we have
\[
\sum_{a \ne 0} \tilde{\psi}_1(au,z) = \frac{1}{2\sqrt{2}\pi} \sum_{a \ne 0} e^{-\pi a^2/t^2} \frac{dw_3}{t} =  \frac{1}{2\sqrt{2}\pi} \left(\sum_{k \in \Z} e^{-\pi t^2k^2} \right) dw_3 -  \frac{1}{2\sqrt{2}\pi} \frac{dw_3}{t},
\]
which goes to $ \tfrac{1}{2\sqrt{2}\pi} dw_3 = \tilde{\psi}_{0,1}(0)$. 
This proves the proposition.
\end{proof}

\subsection{Main result}

In the previous sections, we constructed a closed $2$-form $\theta_{\varphi_2}$ on $\overline{X}$ such
that the restriction of $\theta_{\varphi_2}$ to the boundary $\partial \overline{X}$ was exact with
primitive $ \sum_{[\underline{P}]} \theta^P_{\phi_{0,1}}$. From now on we usually write $\varphi$ for $\varphi_2$ and $\phi$ for $\phi_{0,1}$ if it does not cause any confusion. By the definition of the differential for the mapping cone complex $C^{\bullet}$ we immediately obtain by Theorem~\ref{restriction} and Theorem~\ref{globalexact}

\begin{proposition}
The pair $(\theta_{\varphi_2}(\calL_V), \sum_{[P]} \theta^P_{\phi_{0,1}}(\calL_{W_P}))$ is a $2$-cocycle in $C^{\bullet}$.
\end{proposition}

We write for short $(\theta_{\varphi},\theta_{\phi})$. We obtain a class $[[\theta_{\varphi},\theta_{\phi}]]$ in $H^2(C^{\bullet})$ and hence a class $[\theta_{\varphi},\theta_{\phi}]$ in $H^2_c(X)$. The pairing with $[\theta_{\varphi}, \theta_{\phi}]$ then defines a lift $\Lambda^c$ on differential $2$-forms on $\overline{X}$, which factors through $H^2(\overline{X}) = H^2(X)$. By Lemma~\ref{integralformula} it is given by 
\[
\Lambda^c(\eta,\tau) =  \int_{\overline{X}} \eta \wedge \theta_{\varphi_2} - \sum_{[P]}  \int_{e'(P)} i^*\eta \wedge \theta^P_{\phi_{0,1}}.
\]

\begin{theorem}\label{La^c-hol}
The class $[[\theta_{\varphi},  \theta_{\phi}]]$ is holomorphic, that is, 
\[
L\left(\theta_{\varphi}, \theta_{\phi}\right) = d(\theta_{\psi_1},0).
\]
Hence $[\theta_{\varphi}, \theta_{\phi}]$ is a holomorphic modular form with values in the compactly supported cohomology of $X$, so that the lift $\Lambda^c$ takes values in the holomorphic modular forms.
\end{theorem}
\begin{proof}
By Theorem~\ref{restriction} and Theorem~\ref{globalholomorphic2} we calculate 
\[
d(\theta_{\psi_1},0) = (d\theta_{\psi_1}, i^{\ast} \theta_{\psi_1}) = \left(L \theta_{\varphi_2},  \sum_{[P]}\theta^P_{\psi_{0,1}}\right) =L\left(\theta_{\varphi_2}, \sum_{[P]} \theta^P_{\phi_{0,1}}\right). \qedhere
\]
\end{proof}

It remains to compute the Fourier expansion in $\tau$ of $[\theta_{\varphi}, \theta_{\phi}](\tau)$. We will carry this out in Section~\ref{currents}. 

\begin{theorem}\label{FM-main-th}
We have 
\[
[\theta_{\varphi}, \theta_{\phi}](\tau) =  -\frac{1}{2\pi}\delta_{h0} [\omega] + \sum_{n>0} \PD[C^c_n] q^n \in H_c^2(X,\Q) \otimes M_2(\G(N)).
\]
That is, for any closed $2$-form $\eta$ on $\overline{X}$
\[
\Lambda^c(\eta,\tau) = -\frac{1}{2\pi}\delta_{h0} \int_X \eta \wedge \omega + \sum_{n>0} \left( \int_{C^c_n} \eta \right)q^n,
\]
In particular, the map takes values in the holomorphic modular forms and factors through cohomology. We obtain a map
\begin{equation}
\Lambda^c: H^{2}(X) \to M_{2}(\G(N))
\end{equation}
from the cohomology with compact supports to the space of holomorphic modular forms of weight $2$ for the principal congruence subgroup $\G(N) \subseteq \SL_2(\Z))$.
Alternatively, for $C$ any relative $2$-cycle in $X$ defining a class in $H_2(\overline{X},\partial \overline{X},\Z)$, we have
\[
\Lambda^c(C,\tau) = -\frac{1}{2\pi}\delta_{h0} \vol(C) + \sum_{n>0} ( C^c_n \cdot C ) q^n \in M_2(\G(N)).
\]
\end{theorem}

\begin{remark}
In the theorem we now consider the K\"ahler form $\omega$ representing a class in the compactly supported cohomology. In fact, our mapping cone construction gives an explicit coboundary by which $\omega$ is modified to become rapidly decreasing. 
\end{remark}

\subsection{The Hirzebruch-Zagier Theorem} 

We now view $[\theta_{\varphi}, \theta_{\phi}]$ as a class in $H^2(\tilde{X})$ via the map $j_{\#}: H_c^2(X) \to H^2(\tilde{X})$. We recover the Hirzebruch-Zagier-Theorem. 

\begin{theorem}\label{HZTheorem}
We have 
\[
j_{\#}[\theta_{\varphi}, \theta_{\phi}](\tau) =  -\frac{1}{2\pi}\delta_{h0} [\omega] + \sum_{n>0} [T^c_n] q^n \in H^2(\tilde{X},\Q) \otimes M_2(\G(N)).
\]
In particular,
\[
 -\frac{1}{2\pi}\delta_{h0} \vol(T_m) + \sum_{n>0} (T_n^c \cdot T_m)_{\tilde{X}} q^n \in M_2(\G(N)).
\]
This is the result Hirzebruch-Zagier proved for certain Hilbert modular surfaces (Example~\ref{HZex}) by explicitly computing the intersection numbers $T_m \cdot T^c_n $.
\end{theorem}

\begin{proof}

This follows from Theorem~\ref{FM-main-th} since $j_{\ast} C_n^c = T_n^c$ (Proposition~\ref{CnTn}), combined with the following general principle.
Suppose $\omega$ is a compactly supported form on $X$ such that the cohomology class of $\omega$ is the Poincar\'e dual of the homology class of a cycle $C$: $[\omega] = \PD(C)$. Then we have $ j_{\#}[\omega] = \PD( j_* C)$.
To see this we have only to replace $\omega$ by a cohomologous  `Thom representative' of $\PD(C)$, namely a closed form $\tilde{\omega}$ supported in a tubular neighborhood $N(C)$ of $C$ in $X$ such that the integral of $\tilde{\omega}$ over any disk of $N(C)$ is one. Then it is a general fact from algebraic topology (extension by zero of a Thom class)  that $\tilde{\omega}$ represents the Poincar\'e dual of $C$ in any manifold $M$ containing $N(C)$, in particular for $M = \tilde{X}$.
\end{proof}

\begin{remark}
If one is only interested in recovering the statement of this theorem, then there is also a different way of deriving this from the Kudla-Millson theory. Namely, the lift $\Lambda$ on $H_2(X)$ (Theorem~\ref{KM90}) factors through the quotient of $H_2(X)$ by $H_2(\partial X)$ since the restriction of $\theta_{\varphi_2}$ is exact (Theorem~\ref{globalexact}). But by Proposition~\ref{intersectionhom} we have $j_{\ast} H_2(X) \simeq H_2(X)/  H_2(\partial X)$, and the Hirzebruch-Zagier result exactly stipulates the modularity of the lift of classes in $j_{\ast} H_2(X)$. However, in that way one misses the  remarkable  extra structure coming from $\partial X$ as we will explain in the next subsection.
\end{remark}

\subsection{The lift of special cycles}\label{special-lift-section}

We now consider the lift of a special cycle $C_y$. By Theorem~\ref{FM-main-th} and Lemma~\ref{integralformula} we see
\begin{align}\label{special-lift}
\La^c(C_y,\tau,\calL_V) &= -\frac{1}{2\pi}\delta_{h0} \vol(C_y) + \sum_{n>0} ( C^c_n \cdot C_y) q^n \\
&= \int_{C_y} \theta_{\varphi_2}(\tau,\calL_V) - \sum_{[P]}\int_{(\partial C_y)_P} \theta^P_{\phi_{0,1}}(\tau,\calL_{W_P}). \notag
\end{align}

The two terms on the right, the integrals over $C_y$ and ${\partial C_y}$, are both non-holomorphic modular forms (see below) whose difference is holomorphic (by Theorem~\ref{La^c-hol}). So the generating series series of $(C^c_n \cdot C_y)$ is the sum of two non-holomorphic modular forms. We now give geometric interpretations for the two individual non-holomorphic forms. 

Following \cite{HZ} we define the {\it interior} intersection number of two special cycles by
\[
( C_n \cdot C_y )_X =  (C_n \cdot C_y )^{tr} + \vol(C_n \cap C_y),
\]
the sum of the transversal intersections and the volume of the $1$-dimensional (complex) intersection of  $C_n$ and $C_y$ which occur if one of the components of $C_n$ is equal to $C_y$. 

\begin{theorem}\label{interior-lift}
We have
\[
 \int_{C_y} \theta_{\varphi_2}(\tau,\calL_V) = -\frac{1}{2\pi}\delta_{h0} \vol(C_y) \:+\; 
 \sum_{n=1}^{\infty} ( C_n \cdot C_y )_X q^n \; + \; \sum_{n \in \Q} \sum_{[P]} \int_{(\partial C_y)_{P}}  {\tilde{\psi}_{0,1}^P}(n)(\tau).
 \]
 So the Fourier coefficients of the holomorphic part of the non-holomorphic modular form $\int_{C_y} \theta_{\varphi_2}$ are the interior intersection numbers of the cycles $C_y$ and $C_n$. 
\end{theorem}

\begin{proof}
This is essentially \cite{FCompo}, section~5, where more generally $\Orth(p,2)$ is considered. 
There the interpretation of the holomorphic Fourier coefficients as interior intersection number is given. (For more details of an analogous calculation see \cite{FMspec}, section~8). A little calculation using the formulas in \cite{FCompo} gives the non-holomorphic contribution. A more conceptual proof would use the relationship between $\varphi_2$ and $\tilde{\psi}_1$ (see Proposition~\ref{schluesselV} and Section~\ref{currents}) and the restriction formula for $\tilde{\psi}_1(n)$ (Theorem~\ref{psitilderes}). 
\end{proof}

By slight abuse of notation we write $\Lk(C_n,C_y) = \sum_{[P]} \Lk((\partial C_n)_P, (\partial C_y)_P)$ for the total linking number of $\partial C_n$ and $\partial C_y$. Then by Theorem~\ref{xi'-integralP} we obtain

\begin{theorem}\label{xi'-integral}
\[
\sum_{[P]}\int_{(\partial C_y)_P} \theta^P_{\phi_{0,1}}(\tau,\calL_{W_P}) = 
  \sum_{n>0}  \Lk(C_n,C_y) q^n
 \; + \; \sum_{n \in \Q} \sum_{[P]} \int_{(\partial C_y)_{P}}  {\tilde{\psi}_{0,1}^P}(n)(\tau).
\]
 So the Fourier coefficients of the holomorphic part of $\int_{(\partial C_y)_P} \theta^P_{\phi}(\tau,\calL_{W_P})$ are the linking numbers of the cycles $\partial C_y$ and $ \partial C_n$ at the boundary component $e'(P)$.  
 \end{theorem}

\begin{remark}
There is also another ``global'' proof for Theorem~\ref{xi'-integral}. The cycle $C_y$ intersects $e'(P)$ transversally (when pushed inside) and hence also the cap $A_n$. From this it is not hard to see that we can split the intersection number $C_n^c \cdot C_y$ as 
\[
C^c_n \cdot C_y  =  (C_n \cdot C_y)_X  -  \Lk(C_n,C_y).
\]
Hence Theorem~\ref{xi'-integral} also follows from combining \eqref{special-lift} and Theorem~\ref{interior-lift}. 
\end{remark}

Hirzebruch-Zagier also obtain the modularity of the functions given in Theorems~\ref{interior-lift} and \ref{xi'-integral}, but by quite different methods. In particular, they explicitly calculate the intersection number $T^c_n \cdot T_m$. They split the intersection number into the interior part $(T_n \cdot T_m)_X$ and a `boundary contribution' $(T_n \cdot T_m)_{\infty}$ given by 
 \[
 (T_n \cdot T_m )_{\infty} = (T_n \cdot T_m)_{\tilde{X}-X}  - ({T}_m-T_m^c) \cdot ({T}_n -T_n^c).
 \]
Now by Theorem~\ref{HZTheorem} and its proof we have
\[
T^c_n \cdot T_m = C^c_n \cdot C_m.
\]
We have (per definition) $(T_n \cdot T_m)_X = (C_n \cdot C_m)_X$, so Theorem~\ref{interior-lift} gives the generating series for $(T_n \cdot T_m)_X$. Note that Theorem~5.4 in \cite{FCompo} also compares the explicit formulas in \cite{HZ} for $(T_n \cdot T_m)_X$ with the ones obtained via  $\int_{C_y} \theta_{\varphi_2}(\tau,\calL_V)$. All this implies
\[
 (T_n \cdot T_m)_{\infty}  =\Lk( C_n \cdot C_m).
 \]
Independently, we also obtain this from comparing the explicit formulas for the boundary contribution in \cite{HZ}, Section~1.4 with our formulas for the linking numbers, Theorem~\ref{LinkCnCm} and Example~\ref{LinkCnCmex}.

\section{A current approach for the special cycles}\label{currents}

In this section we prove Theorem~\ref{FM-main-th}, the crucial Fourier coefficient formula for our lift $\Lambda^c$.  As a consequence of our approach we will also obtain Theorem~\ref{linking-dual}, the linking number interpretation for the lift at the boundary. 

\subsection{A differential character for $C_n^c$}

The key step for the entire Kudla-Millson theory is that for $n>0$ the form $\varphi_2(n)$ is a Poincar\'e dual form for the cycle $C_n$, i.e., 

\begin{theorem}[\cite{KM2,KMCan}]
Let $\eta$ be a closed rapidly decreasing $2$-form. Then 
\[
\int_X \eta \wedge  \varphi_2(n)= \left(\int_{C_n} \eta \right) e^{-2\pi n}.
\]
\end{theorem}

To show this they employ at some point a homotopy argument which requires $\eta$ to be rapidly decaying. Since we require $\eta$ to be any closed $2$-form on the compactification $\overline{X}$, their approach is not applicable in our case. Instead, we use a differential character argument for $\varphi_2$ which implicitly already occurred in \cite{BFDuke}, Section~7 for general signature $(p,q)$. Namely, we have

\begin{theorem} (\cite{BFDuke}, Section~7)\label{BrFu}
Let $n>0$. 
The singular form $ \tilde{\psi}_1(n)$ is a differential character in the sense of Cheeger-Simons for the cycle $C_n$. More precisely, $\tilde{\psi}_1(n)$ is a locally integrable $1$-form on $X$, and for any compactly supported $2$-form $\eta$ we have 
\[
\int_{X} \eta \wedge \varphi_2(n)  = \left(\int_{C_n}  \eta \right) e^{-2 \pi n}  - \int_{X}  d \eta\wedge \tilde{\psi}_1(n).
\]
\end{theorem}

\begin{proof}
This is the content of the proofs of Theorem~7.1 and Theorem~7.2 in \cite{BFDuke}. There the analogous properties for a singular theta lift associated to $\psi$ is established. However, the proofs boil down to establish the claims for $\tilde{\psi}_1$. The form $\tilde{\psi}$ there is indeed the form $\tilde{\psi}_1$ of this paper. 
\end{proof}

\begin{remark}\label{Kudla-xi}
The form $\tilde{\psi}_1$ is closely related to Kudla's Green function $\xi$ \cite{KAnn97,KBforms} (more generally for $\Orth(p,2)$) which is given by 
\[
\xi(x) =   \left( \int_1^{\infty} \varphi_0^0(\sqrt{r}x)  \frac{dr}{r} \right) e^{- \pi (x,x) }.
\]
Then $\Xi(n) = \sum_{x\in\calL_n} \xi(x)$ gives rise to a Green's function for the divisor $C_n$ and moreover $dd^c \xi = \varphi_2$. Here $d^c = \tfrac{1}{4\pi i}(\partial - \overline{\partial})$. This suggests $d^c \xi = \tilde{\psi}_1$, which indeed follows from $d^c \varphi_0 = -\psi_1$, see \cite{BFDuke}, Remark~4.5. 
\end{remark}

For $n \in \Q$ we define
\[
 \varphi_2^c(n) :=  \varphi_2(n) - \sum_{[P]} d(f \pi^{\ast} \phi^P_{0,1}(n))
\]
and follow the current approach to show that for $n>0$ the form $\varphi_2^c(n)$
is a Poincar\'e dual form for the cycle $C_n^c$. Here we follow the notation of subsection~\ref{mappingconesection}. That is, $\pi^{\ast} \phi^P_{0,1}(n)$ is the pullback to a product neighborhood $V$ of $\partial \overline{X}$, and $f$ is a smooth function on $V$ of the geodesic flow coordinate $t$ which is $1$ near $t=\infty$ and zero else. Note that  $\varphi_2^c(n)$ is exactly the $n$-th Fourier coefficient of the mapping cone element $[\theta_{\varphi},\theta_{\phi}]$, when realized as a rapidly decreasing form on $X$. We also define 
\[
\tilde{\psi}_1^c(n) = \tilde{\psi}_1(n) - f \pi^{\ast} \phi^P_{0,1}(n). 
\]
We call a differential form $\eta$ on $\overline{X}$ special if in a neighborhood of each boundary component $e'(P)$ it is the pullback of a form $\eta_P$ on $e'(P)$ under the geodesic retraction and if the pullback of the form $\eta_P$ to the universal cover $e(P)$ is $N$-left-invariant. The significance of the forms lies in the fact that the complex of special forms also computes the cohomology of $\overline{X}$. Note that the proof of Theorem~\ref{restriction} shows that $\theta_{\varphi_2}$ is `almost' special; it only differs from a special form by a rapidly decreasing form.

\begin{theorem}\label{newcurrenteq}
Let $n>0$. The form $ \tilde{\psi}_1^c(n)$ is a differential character for the cycle $C^c_n$. More precisely, $\tilde{\psi}_1^c(n)$ is a locally integrable $1$-form on $X$ and satisfies the following current equation on special $2$ forms on $\overline{X}$:
\[
d[\tilde{\psi}_1^c(n)] + \delta_{C_n}  e^{-2\pi n} = [\varphi_2^c(n)].
\]
That is, for any special $2$-form $\eta$ on $\overline{X}$ we have 
\[
\int_{X} \eta \wedge \varphi^c_{2}(n)  = \left(\int_{C^c_n}  \eta \right) e^{-2 \pi n} - \int_{X}  d\eta \wedge \tilde{\psi}^c_{2}(n).
\]
\end{theorem}

This implies Theorem~\ref{FM-main-th} for the positive Fourier coefficients. For $n\leq 0$, the form $\varphi^c_{2}(n)$ is exact with primitive $\tilde{\psi}^c_{2}(n)$ which by Theorem~\ref{psitilderes} is decaying. So Theorem~\ref{newcurrenteq} holds also for $n \leq 0$ with $C_n^c = \emptyset$. Hence for the these coefficients only the term $x=0$ contributes, which gives the integral of $\eta$ against the K\"ahler form. 

\begin{remark}\label{Kudla-modification}
In view of Remark~\ref{Kudla-xi} it is very natural question to ask how one can modify Kudla's Green's function $\Xi(n)$ to obtain a Green's function for the cycle $T_n^c$ in $\tilde{X}$. Extensive discussions with K\"uhn suggest that (if $X$ has only one cusp) 
\[
\Xi(n) - t \sum_{\substack{x \in \calL_W\\ (x,x)=2n}} f \pi^{\ast}(B(x)+B'(x))
\]
is such a Green's function, but we have not checked all details.
\end{remark}

\subsection{Proof of Theorem~\ref{newcurrenteq} }\label{8.1}

For simplicity assume that $X$ has only one cusp and continue the drop the superscript $P$. We let $\rho_{T}$ be a family of smooth functions on a standard fundamental domain $\calF$ of $\G$ in $D$ only depending on $t$ which is $1$ for $t\leq T$ and $0$ for $T+1$. We then have 
\begin{align*}
\int_{X} \eta \wedge \varphi^c_{2}(n)  &= \lim_{T\to \infty} \int_{X} \rho_T \eta \wedge
\left(\varphi_2(n) -  d(f \pi^{\ast} \phi_{0,1}(n)) \right). 
\end{align*}
We apply Theorem~\ref{BrFu} for the compactly supported form $\rho_T\eta$ and obtain
\begin{align}\label{eq1}
\int_{X} \eta \wedge \varphi^c_{2}(n&)=  \lim_{T\to \infty} \Biggl[ \left(\int_{C_n}  \rho_T \eta \right) e^{- 2\pi n} - \int_{X}  d (\rho_T \eta) \wedge \tilde{\psi}_1(n) \\
& \quad - \int_X  d\left( \rho_T \eta \wedge (f \pi^{\ast} \phi_{0,1}(n)) \right) - d(\rho_T \eta) \wedge f \pi^{\ast} \phi_{0,1}(n) \Biggr] \notag
\end{align}
The first term on the right hand side of \eqref{eq1} goes to $\left(\int_{C_n} \eta\right)e^{-2\pi n}$ as $T \to \infty$, while the third vanishes for any $T$ by Stokes' theorem. For the two remaining terms of \eqref{eq1} we first note $d(\rho_T \eta) = \rho_T'(t) dt \wedge \eta + \rho_T d\eta$ and  $\rho_T'(t)=0$ outside $[T,T+1]$. We obtain for these two terms
\begin{multline}\label{eq2}
-  \int_{X}  (d \eta) \wedge \left( \tilde{\psi}_1(n) - f \pi^{\ast} \phi_{0,1}(n) \right) \\ - \lim_{T\to \infty} \int_T^{T+1} \int_{e'(P)} \rho_T'(t)dt \wedge \eta \wedge \left( \tilde{\psi}_1(n) - f  \pi^{\ast}\phi_{0,1}(n)\right). 
\end{multline}
It remains to compute the second term in the previous equation. For $T$ sufficiently large we have $f \equiv 1$. Furthermore by Theorem~\ref{psitilderes} and its proof we have $\tilde{\psi}_1(n) = \pi^{\ast} \tilde{\psi}_{0,1}(n) + O(e^{-Ct})$. As
$\phi_{0,1}(n) = \tilde{\psi}_{0,1}(n)+\tilde{\psi}'_{0,1}(n)$, we can replace  
$\tilde{\psi}_1(n) - f  \pi^{\ast}\phi_{0,1}(n)$ by $-\pi^{\ast} \tilde{\psi'}_{0,1}(n)$. Since $\eta$ is special it does not depend on the $t$-variable near the boundary. For the last term in \eqref{eq2}
\[
 \lim_{T\to \infty} \int_T^{T+1}  \rho_T'(t)dt \int_{e'(P)} \eta \wedge \pi^{\ast} \tilde{\psi'}_{0,1}(n) = -  \int_{e'(P)} \eta \wedge \tilde{\psi'}_{0,1}(n) = -
\left(\int_{A_n} \eta \right) e^{-2\pi n}.
\]
Indeed, for $\eta = \Omega$ this is Remark~\ref{youwillneedthis}. Otherwise, $\eta$ is exact with special primitive $\omega$, and it is not hard to see that the proof of Proposition~\ref{finalintegral} carries over to this situation. Since $C_n^c = C_n \coprod (-A_n)$ collecting all terms completes the proof of Theorem~\ref{newcurrenteq}.

\end{document}